\documentclass[11pt]{amsart}
\usepackage[totalwidth=440pt, totalheight=640pt]{geometry}
\usepackage{amsmath, amssymb, amsthm, amsfonts, mathrsfs, eucal, enumerate, natbib, layout}
\usepackage[all,tips,2cell]{xy}
\usepackage[usenames]{color}
\usepackage{xspace}
\usepackage{ifpdf}
\usepackage{ifthen}
\usepackage{hyperref}
\usepackage{tikz}
\usepackage{verbatim}
\usetikzlibrary{matrix,arrows}
\usepackage{mathtools}
\newtheorem{theorem}{Theorem}[section]
\newtheorem*{theorem*}{Theorem}
\newtheorem{corollary}[theorem]{Corollary}
\newtheorem*{corollary*}{Corollary}
\newtheorem{lemma}[theorem]{Lemma}

\newtheorem*{claim*}{Claim}
\newtheorem*{lemma*}{Lemma}
\newtheorem{proposition}[theorem]{Proposition}
\newtheorem*{proposition*}{Proposition}
\newtheorem{def-prop}[theorem]{Definition-Proposition}
\theoremstyle{definition}
\newtheorem{definition}[theorem]{Definition}
\newtheorem*{definition*}{Definition}
\newtheorem{remark}[theorem]{Remark}

\newtheorem{example}[theorem]{Example}
\newtheorem*{example*}{Example}
\numberwithin{equation}{section}

\newcommand{\id}{\mathrm{id}}
\newcommand{\ass}{\mathsf{a}}
\newcommand{\st}{\mathrm{st}}

\newcommand{\eext}{\mathrm{Ext}}
\newcommand{\Ext}{\mathcal{E}\mathrm{xt}}
\newcommand{\pr}{\mathrm{pr}}
\newcommand{\inc}{\mathrm{inc}}

\newcommand{\Hom}{\mathbb{H}\mathrm{om}}
\newcommand{\HomExt}{\mathbb{H}\mathrm{om}_{\mathcal{E}\mathrm{xt}}}
\newcommand{\hhom}{\mathrm{Hom}}
\newcommand{\MC}{\mathrm{MC}}

\newcommand{\invass}{\mathsf{a}^{-1}}
\newcommand{\comm}{\mathsf{c}}

\newcommand{\TORS}{\textsc{Tors}}

\newcommand{\PICARD}{\mathcal{P}\textsc{icard}(\mathbf{S})}
\newcommand{\FPICARD}{\mathcal{P}\textsc{icard}^{\flat \flat}(\mathbf{S})}

\DeclareMathOperator{\coker}{coker}
\DeclareMathOperator{\fcirc}{{\diamond}}
\DeclareMathOperator{\ra}{{\rightarrow}}
\DeclareMathOperator{\lra}{{\longrightarrow}}
\DeclareMathOperator{\Ra}{{\Rightarrow}}

\DeclareMathOperator{\Rra}{{\Rrightarrow}}
\DeclareMathOperator{\Lla}{{\Lleftarrow}}

\newcommand{\boyut}[1]{\mathscr{#1}}
\newcommand{\oneA}{\boyut A}

\newcommand{\oneU}{\boyut U}

\newcommand{\ES}{\mathbf{S}}

\newcommand{\Ker}{\mathscr{K}\mspace{-4mu}\mathit{er}}

\newcommand{\Aut}{\mathscr{A}\mspace{-4mu}\mathit{ut}}

\newcommand{\twocat}[1]{\mathcal{#1}^{[-2,0]}(\mathbf{S})}
\newcommand{\twoA}{\mathbb{A}}
\newcommand{\twoB}{\mathbb{B}}
\newcommand{\twoC}{\mathbb{C}}
\newcommand{\twoD}{\mathbb{D}}
\newcommand{\twoE}{\mathbb{E}}
\newcommand{\twoF}{\mathbb{F}}

\newcommand{\twoK}{\mathbb{K}}
\newcommand{\twoO}{\mathbb{O}}
\newcommand{\twoL}{\mathbb{L}}

\newcommand{\twoKer}{\mathbb{K}\mathrm{er}}
\newcommand{\twoCo}{\mathbb{C}\mathrm{oker}}
\newcommand{\twoP}{\mathbb{P}}
\newcommand{\twoQ}{\mathbb{Q}}

\newcommand{\twocatD}{\twocat D}
\newcommand{\twocatK}{\twocat K}
\newcommand{\twocatT}{\twocat T}
\newcommand{\twocatC}{\twocat C}

\newcommand{\der}{\mathcal{D}(\mathbf{S})}
\newcommand{\Frac}{\mathtt{FRAC}}
\newcommand{\kesir}[2]{\genfrac{}{}{0pt}{}{#1}{#2}}
\xyoption{2cell}

\begin{document}
\UseAllTwocells
\title[Extensions and Picard 2-Stacks]
{Extensions of Picard 2-Stacks and \\
the cohomology groups ${\mathrm{Ext}}^i$ of length 3 complexes}
\author{Cristiana Bertolin}
\address{Dipartimento di Matematica, Universit\`a di Torino, Via Carlo Alberto 10, 
Italy}
\email{cristiana.bertolin@googlemail.com}
\author{Ahmet Emin Tatar}
\address{Department of Mathematics and Statistics, KFUPM, Dhahran, KSA}
\email{atatar@kfupm.edu.sa}

\subjclass{18G15}
\keywords{length 3 complexes, Picard 2-stacks, extensions, cohomology groups ${\mathrm{Ext}}^i$}
\let\thefootnote\relax\footnotetext{Second author is supported by KFUPM under research grant JF101015}
\date{}


\begin{abstract} 
The aim of this paper is to define and study the 3-category of extensions of Picard 2-stacks over a site $\ES$ and  to furnish a geometrical description of the cohomology groups $\eext^{i}$ of length 3 complexes of abelian sheaves. More precisely, our main Theorem furnishes 
\begin{enumerate}
	\item a parametrization of the equivalence classes of objects, 1-arrows, 2-arrows, and 3-arrows of the 3-category of extensions of Picard 2-stacks by the cohomology groups $ {\mathrm{Ext}}^{i}$, and
	\item a geometrical description of the cohomology groups ${\mathrm{Ext}}^{i}$ of length 3 complexes of abelian sheaves via extensions of Picard 2-stacks.
\end{enumerate}
To this end, we use the triequivalence between the 3-category $2\PICARD$ of Picard 2-stacks and the tricategory $\twocatT$ of length 3 complexes of abelian sheaves over $\ES$ introduced by the second author in \cite{MR2735751}, and we define the notion of extension in this tricategory $\twocatT$, getting a pure algebraic analogue of the 3-category of extensions of Picard 2-stacks. The calculus of fractions that we use to define extensions in the tricategory $\twocatT$ plays a central role in the proof of our Main Theorem. 
\end{abstract} 


\maketitle
\tableofcontents

\section*{Introduction}
 Let $\ES$ be a site. A \emph{Picard $\ES$-2-stack} $\twoP$ is an $\ES$-2-stack in 2-groupoids  equipped with a morphism of 2-stacks
$\otimes: \twoP \times \twoP \lra \twoP$ expressing the group law and two natural 2-transformations $\ass$ and $\comm$ expressing the associativity and commutativity constraints for the group law $\otimes$, such that for any object $U$ of $\ES$, $\twoP(U)$ is a Picard 2-category (i.e. it is possible to make the sum of two objects of $\twoP(U)$ and this sum is associative and commutative). Picard 2-stacks form a 3-category $2\PICARD$ whose hom-2-groupoid consists of additive 2-functors, morphisms of additive 2-functors and modifications of morphisms of additive 2-functors.

As Picard $\ES$-stacks are the categorical analogues of length 2 complexes of abelian sheaves over $\ES$, the concept of Picard $\ES$-2-stacks is the categorical analogue of length 3 complexes of abelian sheaves over $\ES$. In fact in \cite{MR2735751}, the second author proves the existence of an equivalence of categories
\begin{equation*}\label{intro:2st_flat_flat}
2\st^{\flat\flat} \colon \xymatrix@1{\twocatD \ar[r] & 2\FPICARD ,}
\end{equation*} 
between the full subcategory $\twocatD$ of the derived category $\der$ of complexes of abelian sheaves over $\ES$ such that $\mathrm{H}^{-i}(A) \neq 0$ for $i=0,1,2$, and the category of Picard 2-stacks $2\FPICARD$ obtained form the 3-category $2\PICARD$ by taking as objects the Picard 2-stacks and as arrows the equivalence classes of additive 2-functors, i.e. the 2-isomorphism classes (up to modifications) of additive 2-functors (remark that morphisms of additive 2-functors are not strictly invertible, but just invertible up to modifications). We denote by $[\,\,]^{\flat\flat}$ the inverse equivalence of $2\st^{\flat\flat}$. This equivalence of categories $2\st^{\flat\flat}$  generalizes to Picard 2-stacks Deligne's result for Picard stacks \cite[Proposition 1.4.15]{SGA4}.

In this paper we define and study extensions of Picard 2-stacks. If $\twoA$ and $\twoB$ are two Picard 2-stacks over $\ES$, an \emph{extension} of $\twoA$ by $\twoB$ consists of a Picard 2-stack $\twoE$, two additive 2-functors $I:\twoB \ra \twoE$ and $J:\twoE \ra \twoA$, a morphism of additive 2-functors $J \circ I \Rightarrow 0$, such that the following equivalent conditions are satisfied:
\begin{itemize}
	\item $\pi_0(J): \pi_0(\twoE) \ra \pi_0(\twoA)$ is surjective and $I$ induces an equivalence of Picard 2-stacks between $\twoB$ and $\twoKer(J);$ 
	\item $\pi_2(I): \pi_2(\twoB) \ra \pi_2(\twoE)$ is injective and $J$ induces an equivalence of Picard 2-stacks between $\twoCo(I)$ and $\twoA$.
\end{itemize}
The extensions of $\twoA$ by $\twoB$ form a 3-category $\Ext(\twoA,\twoB)$ where
 the objects are extensions of $\twoA$ by $\twoB$, the 1-arrows are morphisms of extensions,
 the 2-arrows are 2-morphisms of extensions and the 3-arrows are 3-morphisms of extensions (see Definitions \ref{def:extpic2stacks}, \ref{def:morextpic2stacks}, \ref{def:2morextpic2stacks}, \ref{def:3morextpic2stacks}).

Although regular morphisms of length 3 complexes of abelian sheaves induce additive 2-functors between Picard 2-stacks, not all of them are obtained in this way. In order to resolve this problem, in \cite{MR2735751} the second author introduces the tricategory $\twocatT$ 
of length 3 complexes of abelian sheaves over $\ES$, in which arrows between length 3 complexes are fractions, and he showes that there is 
a triequivalence 
\begin{equation*}
2\st \colon \xymatrix@1{\twocatT \ar[r] & 2\PICARD ,}
\end{equation*}
between the tricategory $\twocatT$ and the 3-category $2\PICARD$ of Picard 2-stacks. In this paper, we define also
the notion of extension of length 3 complexes in the tricategory $\twocatT$: If $A$ and $B$ be two length 3 complexes of abelian sheaves over the site $\ES$, an  \emph{extension} of $A$ by $B$ consists of a length 3 complex of abelian sheaves $E$, two fractions $i=(q_i,M,p_i): B \ra E$ and $j=(q_j,N,p_j): E \ra A$, a 1-arrow of fractions $R=(r,R,r'):j \fcirc i \Ra 0$, such that the following equivalent conditions are satisfied:
\begin{itemize}
	\item ${\mathrm{H}}^{0}(p_j) \circ ({\mathrm{H}}^{0}(q_j))^{-1}: {\mathrm{H}}^{0}(E) \ra {\mathrm{H}}^{0}(A)$ is surjective and $i$ induces a quasi-isomorphism between $B$ and $ \tau_{\leq 0} (\MC(p_j)[-1])$;
	\item  ${\mathrm{H}}^{-2}(p_i)\circ ({\mathrm{H}}^{-2}(q_i))^{-1} \colon {\mathrm{H}}^{-2}(B) \ra {\mathrm{H}}^{-2}(E)$ is injective and $j$ induces a quasi-isomorphism between $ \tau_{\geq -2}(\MC(p_i))$ and $A$,
\end{itemize}
where $\fcirc$ represents the fraction composition. 

The extensions of $A$ by $B$ in $\twocatT$ form a tricategory $\Ext(A,B)$ where the objects are extensions of $A$ by $B$,
 the 1-arrows are morphisms of extensions, the 2-arrows are 2-morphisms of extensions and the 3-arrows are 3-morphisms of extensions (see Definitions \ref{def:extension_via_fractions}, \ref{def:morphism_of_extensions_via_fractions}, \ref{def:2-morphism_of_extensions_via_fractions}, \ref{def:3-morphism_of_extensions_via_fractions}).
The tricategory $\Ext(A,B)$ is the pure algebraic analogue of the 3-category $\Ext(\twoA,\twoB)$ of extensions of Picard 2-stacks.

We introduce the notions of product, fibered product, called also pull-back, and fibered sum, called also push-down, of Picard 2-stacks (resp. of length 3 complexes). Remark that when we define the fibered product (resp. the fibered sum) of length 3 complexes we are actually computing certain homotopy limits (resp. colimits) of complexes by using the equivalence between such complexes and Picard 2-stacks.

We define the following groups:
\begin{itemize}
\item $\Ext^{1}(\twoA,\twoB)$ is the group of equivalence classes of objects of $\Ext(\twoA,\twoB)$;
\item $\Ext^{0}(\twoA,\twoB)$ is the group of 2-isomorphism classes of morphisms of extensions from an extension $\twoE$ of $\twoA$ by $\twoB$ to itself;
\item $\Ext^{-1}(\twoA,\twoB)$ is the group of 3-isomorphism classes of 2-automorphisms of morphisms of extensions from $\twoE$ to itself; and finally
\item $\Ext^{-2}(\twoA,\twoB)$ is the group of 3-automorphisms of 2-automorphisms of morphisms of extensions from $\twoE$ to itself.
\end{itemize}

The group structure on the $\Ext^{i}(\twoA,\twoB)$ for $i=1,0,-1,-2$ is defined in the following way: 
Using pull-backs and push-downs of Picard 2-stacks, we introduce the notion of sum of two extensions of $\twoA$ by $\twoB$ which
furnishes the abelian group structure on $\Ext^{1}(\twoA,\twoB)$. The 2-stack $\HomExt(\twoE,\twoE)$ of morphisms of extensions from an extension $\twoE$ of $\twoA$ by $\twoB$ to itself is endowed with a Picard structure and so its homotopy groups $\pi_{i}(\HomExt(\twoE,\twoE))$ for $i=0,1,2$ are abelian groups. Since by definition 
$$\Ext^{-i}(\twoA,\twoB) = \pi_{i}(\HomExt(\twoE,\twoE))$$
we have that the $\Ext^{i}(\twoA,\twoB)$ for $i=0,-1,-2$ are abelian groups.

We can finally state our main Theorem which can be read from left to right and from right to left furnishing respectively 
\begin{enumerate}
	\item a parametrization of the elements of $\Ext^{i}(\twoA,\twoB)$ by the cohomology groups $ {\mathrm{Ext}}^{i}\big([\twoA]^{\flat\flat},[\twoB]^{\flat\flat}\big)$, and so in particular a parametrization of the equivalence classes of extensions of $\twoA$ by $\twoB$ by the cohomology group $ {\mathrm{Ext}}^{1}\big([\twoA]^{\flat\flat},[\twoB]^{\flat\flat}\big)$;
	\item a geometrical description of the cohomology groups ${\mathrm{Ext}}^{i}$ of length 3 complexes of abelian sheaves via extensions of Picard 2-stacks.
\end{enumerate}

\begin{theorem}\label{thm:introext}
Let $\twoA$ and $\twoB$ be two Picard 2-stacks. Then for $ i=1,0,-1,-2$, we have the following isomorphisms of groups 
	\[\Ext^{i}(\twoA,\twoB) \cong {\mathrm{Ext}}^{i}\big([\twoA]^{\flat\flat},[\twoB]^{\flat\flat}\big)=\hhom_{\der}\big([\twoA]^{\flat\flat},[\twoB]^{\flat\flat}[i]\big).\]
\end{theorem}

The use of the tricategory $\twocatT$, and in particular the use of fractions as arrows between length 3 complexes instead of regular morphisms of complexes, play a central role in the proof of this main Theorem.

Picard 3-stacks are not defined yet. Assuming their existence, the group law that we define for equivalence classes of extensions of Picard 2-stacks should furnish a structure of Picard 3-stack on the 3-category $\Ext(\twoA,\twoB)$.
In this setting our main Theorem \ref{thm:introext} says that the Picard 3-stack $\Ext(\twoA,\twoB)$ is equivalent to the Picard 3-stack associated to the object 
\[\tau_{\leq 0} {\mathrm{R}}\hhom\big([\twoA]^{\flat\flat},[\twoB]^{\flat\flat}[1]\big),\]
 of $\mathcal{D}^{[-3,0]}(\ES)$ via the generalization of the equivalence $2\st^{\flat\flat}$ to Picard 3-stacks and length 4 complexes of abelian sheaves. More generally, we expect that extensions of Picard $n$-stacks of $\mathbf{A}$ by $\mathbf{B}$ build a Picard $(n+1)$-stack which should be equivalent to the Picard $(n+1)$-stack associated to the object $\tau_{\leq 0} {\mathrm{R}}\hhom\big([\mathbf{A}],[\mathbf{B}][1]\big)$ of $\mathcal{D}^{[-(n+1),0]}(\ES)$ via the generalization of the equivalence $2\st^{\flat\flat}$ to Picard $(n+1)$-stacks and length $n+2$ complexes of abelian sheaves.  \\
Moreover, always in the setting of Picard 3-stacks, in order to define the groups $\Ext^i(\twoA,\twoB)$ we could use the 
 homotopy groups $\pi_{i}$ for $i=0,1,2,3$ of the Picard 3-stack $\Ext(\twoA,\twoB)$. In fact we have 
$$\Ext^i(\twoA,\twoB)=\pi_{-i+1}(\Ext(\twoA,\twoB))\qquad \mathrm{for}\; i=1,0,-1,-2.$$

Another consequence of the group law defined on $\Ext^{1}(\twoA,\twoB)$ is that for three Picard 2-stacks $\twoO,\twoA$ and $\twoB$, we have the equivalences of 3-categories
$\Ext(\twoO \times \twoA,\twoB) \cong \Ext(\twoO,\twoB) \times \Ext(\twoA,\twoB)$ and $
\Ext(\twoO,\twoA \times \twoB) \cong \Ext(\twoO,\twoA) \times \Ext(\twoO,\twoB),$
which imply the following decomposition for the cohomological groups ${\mathrm{Ext}}^i$ for $i=1,0,-1,-2$:
\[{\mathrm{Ext}}^i([\twoO]^{\flat\flat} \times [\twoA]^{\flat\flat},[\twoB]^{\flat\flat}) \cong {\mathrm{Ext}}^i([\twoO]^{\flat\flat},[\twoB]^{\flat\flat}) \times {\mathrm{Ext}}^i([\twoA]^{\flat\flat},[\twoB]^{\flat\flat}),\]
\[{\mathrm{Ext}}^i([\twoO]^{\flat\flat},[\twoA]^{\flat\flat} \times [\twoB]^{\flat\flat}) \cong {\mathrm{Ext}}^i([\twoO]^{\flat\flat},[\twoA]^{\flat\flat}) \times {\mathrm{Ext}}^i([\twoO]^{\flat\flat},[\twoB]^{\flat\flat}).\]

All the definitions we have introduced in this paper for Picard 2-stacks and for length 3 complexes of abelian sheaves generalize the classical definitions for abelian groups and abelian sheaves respectively: for example, our definition of pull-back of length 3 complexes reduces to the classical notion of pull-back of abelian sheaves if we consider the special case of length 3 complexes concentrated only in degree 0 (i.e. $A^{-2}=A^{-1}=0$ and $A^0 \not=0$). 

We study also the relations between the homotopy groups $\pi_i$ of the Picard 2-stacks $\twoA,\twoB$ and the homotopy groups $\pi_i$ of the extensions of $\twoA$ by $\twoB$. We get a long exact sequence of abelian sheaves (\ref{long_exact_sequence}) which we see as a confirmation that our definition of extension of Picard 2-stacks works. 

We hope that this work will shed some light on the notions of ``pull-back", ``push-down" and ``extension" for higher categories with group-like operation. In particular we pay a lot of attention to write down the proofs in such a way that they can be easily generalized to Picard n-stacks and length n+1 complexes of abelian sheaves.

The most relevant ancestors of this paper are \cite{bertolin} where the first author studies the homological interpretation of extensions of Picard stacks (i.e. she proves Theorem \ref{thm:introext} for Picard stacks), and \cite{MR1983014} where D. Bourn and E. M. Vitale study extensions of symmetric categorical groups, together with their pull-back, push-down and sum. 

The study of extensions of Picard n-stacks has important applications in the theory of motives: for example, in \cite{bertolin2} the first author uses extensions of Picard stacks in order to prove Deligne's conjecture on extensions of 1-motives (recall that a 1-motive can be seen as a complex of abelian sheaves of length 2).

This paper is organized as follows: In Section 1, we recall some basic definitions and results on the 3-category $2\PICARD$ of Picard 2-stacks. In Section 2, we introduce the notions of product, pull-back and push-down for Picard 2-stacks and for length 3 complexes of abelian sheaves in the tricategory $\twocatT$. In Section 3, we define extensions of Picard 2-stacks, morphisms, 2-morphisms and 3-morphisms of extensions of Picard 2-stacks, getting the 3-category $\Ext_{2\PICARD}$ of extensions of Picard 2-stacks. 
In Section 4, we introduce extensions of length 3 complexes in the tricategory $\twocatT$, morphisms, 2-morphisms and 3-morphisms of such extensions in $\twocatT$, getting the tricategory $\Ext_{\twocatT}$ of extensions of length 3 complexes in $\twocatT$. Section 4 is the algebraic counter part of Section 3: in fact, the triequivalence $2\st$ 
between the tricategory $\twocatT$ and the 3-category $2\PICARD$ induces a triequivalence between $\Ext_{\twocatT}$ and $\Ext_{2\PICARD}$.
Using the results of Section 2, in Section 5 we introduce the notions of pull-back and push-down of extensions of Picard 2-stacks which allow us to define an abelian group law on the set $\Ext^1(\twoA,\twoB)$ of equivalence classes of extensions of Picard 2-stacks. This group law is a categorification of the abelian group law on the set of equivalence classes of extensions of abelian groups, known as the Baer sum. In Section 6, we finally prove our main Theorem.
In Appendix A we get a long exact sequence involving the homotopy groups $\pi_i$ of an extension of Picard 2-stacks. In Appendix B we sketch the proof of the fact that the fibered sum of Picard 2-stacks satisfies the universal property.

\section*{Acknowledgment}
We are very grateful to Enrico Vitale for explaining us the pull-back of Picard 2-stacks, and to Pierre Deligne for pointing out the length 3 complex of abelian sheaves corresponding to this pull-back. We thank also Ettore Aldrovandi for interesting conversations  about the subject.

\section*{Notation}

A \emph{strict 2-category} (just called 2-category) $\mathbb{A}=(A,C(a,b),K_{a,b,c},U_{a})_{a,b,c \in A}$ is given by the following data: a set $A$ of objects $a,b,c, ...$; for each ordered pair $(a,b)$ of objects of $A$, a category $C(a,b)$;
   for each ordered triple $(a,b,c)$ of objects $A$, a composition functor $K_{a,b,c}:C(b,c) \times C(a,b) \lra C(a,c),$
       that satisfies the associativity law;
   for each object $a$, a unit functor $U_a:\boldsymbol 1 \ra C(a,a)$ where $\boldsymbol 1$ is the terminal category, that provides a left and right identity for the composition functor.

This set of axioms for a 2-category is exactly like the set of axioms for a category in which the collection of arrows $\hhom(a,b)$ have been replaced by the categories $C(a,b)$. We call the categories $C(a,b)$ (with $a,b \in A$) the \emph{hom-categories} of the 2-category $\mathbb{A}$: the objects of $C(a,b)$ are the \emph{1-arrows} of $\mathbb{A}$ and the arrows of $C(a,b)$ are the \emph{2-arrows} of $\mathbb{A}$. A \emph{2-groupoid} is a 2-category whose 1-arrows are invertible up to a 2-arrow and whose 2-arrows are strictly invertible.

A \emph{bicategory} is weakened version of a 2-category in the following sense: instead of requiring that the associativity
and unit laws for arrows hold as equations, one requires merely that they hold up to isomorphisms (see \cite{MR0220789}). A \emph{bigroupoid} is a bicategory whose 1-arrows are invertible up to a 2-arrow and whose 2-arrows are strictly invertible. The difference between 2-groupoid and bigroupoid is just the underlying 2-categorical structure: one is strict and the other is weak. 

For more details about 2-categories and for other 2-categorical structures such as 2-functors and natural transformations of 2-functors, we refer to \cite[Chapter 1]{MR0364245}.
 
A \emph{strict 3-category} (just called 3-category) $\mathcal{A}=(A,\mathbb{C}(a,b),K_{a,b,c},U_{a})_{a,b,c \in A}$ is given by the following data: a set $A$ of objects $a,b,c, ...$;
  for each ordered pair $(a,b)$ of objects of $A$, a 2-category $\mathbb{C}(a,b)$;
   for each ordered triple $(a,b,c)$ of objects $A$, a composition 2-functor $K_{a,b,c}:\mathbb{C}(b,c) \times \mathbb{C}(a,b) \longrightarrow \mathbb{C}(a,c),$  that satisfies the associativity law;
   for each object $a$, a unit 2-functor $U_a:\boldsymbol 1 \ra \mathbb{C}(a,a)$ where $\boldsymbol 1$ is the terminal 2-category, that provides a left and right identity for the composition 2-functor.

This set of axioms for a 3-category is exactly like the set of axioms for a category in which the arrow-sets $\hhom(a,b)$ have been replaced by the 2-categories $\mathbb{C}(a,b)$. We call the 2-categories $\mathbb{C}(a,b)$ (with $a,b \in A$) the \emph{hom-2-categories} of the 3-category $\mathcal{A}$: the objects of $\mathbb{C}(a,b)$ are the \emph{1-arrows} of $\mathcal{A}$, the 1-arrows of $\mathbb{C}(a,b)$ are the \emph{2-arrows} of $\mathcal{A}$, and the 2-arrows of $\mathbb{C}(a,b)$ are the \emph{3-arrows} of $\mathcal{A}$. 

A \emph{tricategory} is weakened version of a 3-category in the sense of \cite{MR1261589}. We also use \emph{trifunctor} in the sense of \cite{MR1261589}. A \emph{triequivalence of tricategories} $ T: \mathcal{A} \lra \mathcal{A}'$ is a trifunctor which induces biequivalences $T_{a,b}:\mathbb{A}(a,b) \ra \mathbb{A}'(T(a),T(b))$ between the hom-bicategories for all objects $a,b \in \mathcal{A}$ and such that every object in $\mathcal{A}'$ is biequivalent in $\mathcal{A}'$ to an object of the form $T(a)$ where $a$ is an object in $\mathcal{A}$.

Let $\ES$ be a site. For the notions of $\ES$-pre-stacks, $\ES$-stacks and morphisms of $\ES$-stacks we refer to Chapter II 1.2. of \cite{MR0344253}. An \emph{$\ES$-2-stack in 2-groupoids} $\twoP$ is a fibered 2-category in 2-groupoids over $\ES$ such that for every pair of objects $X,Y$ of the 2-category $\twoP(U)$, the fibered category of morphisms $\mathrm{Arr}_{\twoP(U)}(X,Y)$ of $\twoP(U)$ is an $\ES/U$-stack (called the $\ES/U$-stack of morphisms), and 
	 2-descent is effective for objects in $\twoP$.
See \cite[\S I.3]{MR0364245} and \cite[\S 6]{breen-2006} for more details.

Denote by ${\mathcal{K}}(\ES)$ the category of complexes of abelian sheaves on the site $\ES$: all complexes that we consider in this paper are cochain complexes. Let $\twocatK$ be the subcategory of ${\mathcal{K}}(\ES)$ consisting of complexes $K=(K^i)_{i \in \mathbb{Z}}$ such that $K^i=0$ for $i \not= -2,-1$ or $0$. The good truncation $ \tau_{\leq n} K$ of a complex $K$ of ${\mathcal{K}}(\textbf{S})$ is the complex: $ (\tau_{\leq n} K)^i= K^i$ for $i <n,  (\tau_{\leq n} K)^n= \ker(d^n)$ and $(\tau_{\leq n} K)^i= 0$ for $i > n$. The bad truncation $\sigma_{\leq n} K$ of a complex $K$ of ${\mathcal{K}}(\ES)$ is the complex: $ (\sigma_{\leq n} K)^i= K^i$ for $i \leq n $ and $ (\sigma_{\leq n} K)^i= 0$ for $i > n$. For any $i \in {\mathbb{Z}}$, the shift functor $[i]:{\mathcal{K}}(\ES) \rightarrow {\mathcal{K}}(\ES) $ acts on a  complex $K=(K^n)_{n \in \mathbb{Z}}$ as $(K[i])^n=K^{i+n}$ and $d^n_{K[i]}=(-1)^{i} d^{n+i}_{K}$.

Denote by $\der$ the derived category of the category of abelian sheaves on $\ES$, and let $\twocatD$ be the full subcategory of $\der$ consisting of complexes $K$ such that ${\mathrm{H}}^i (K)=0$ for $i \not= -2,-1$ or $0$. If $K$ and $L$ are complexes of $\der$, the group $\eext^i(K,L)$ is by definition $\hhom_{\der}(K,L[i])$ for any $i \in {\mathbb{Z}}$. Let ${\mathrm{R}}\hhom(-,-)$ be the derived functor of the bifunctor $\hhom(-,-)$. The cohomology groups ${\mathrm{H}}^i\big({\mathrm{R}}\hhom(K,L) \big)$ of ${\mathrm{R}}\hhom(K,L)$ are isomorphic to $\hhom_{\der}(K,L[i])$.

Let $\twocatC$ be the 3-category whose objects are length 3 complexes of abelian sheaves over $\ES$ placed in degree -2,-1,0, and whose hom-2-groupoid  $\hhom_{\twocatC} (K,L)$ is the 2-groupoid associated to $\tau_{\leq 0} (\hhom(K,L))$ (see \S 3.1 \cite{MR2735751} for an explicit description of this 3-category).
 
Denote by $\twocatT$ the tricategory whose objects are length 3 complexes of abelian sheaves over $\ES$ placed in degree -2,-1,0, and whose hom-bigroupoid $\hhom_{\twocatT}(K,L)$ is the bigroupoid $\Frac(K,L)$ defined as follows:
\begin{itemize}
	\item an object of $\Frac(K,L)$ is a triple $(q,M,p): K\stackrel{q}{\leftarrow}M \stackrel{p}{\rightarrow}L$, called \emph{fraction},
	where $M$ is a complex of abelian sheaves, $p$ is a morphism of complexes and $q$ is a quasi-isomorphism;
	\item a 1-arrow between fractions $(q_1,M_1,p_1) \Rightarrow (q_2,M_2,p_2)$, called \emph{1-arrow of fractions}, is a triple $(r,N,s)$ with
	$N$ a complex of abelian sheaves and $r:N \rightarrow M_2, s:N \rightarrow M_1$ quasi-isomorphisms such that $q_1 \circ s =q =q_2 \circ r $ and $p_1 \circ s =p =p_2 \circ r $;
	\item a 2-arrow between 1-arrows of fractions $(r_1,N_1,s_1) \Rra (r_2,N_2,s_2)$, called \emph{2-arrow of fractions}, is an isomorphism of complexes of abelian sheaves $t:N_1 \ra N_2$ such that the diamond diagram (see \cite[(4.2)]{MR2735751}) commutes.
\end{itemize}
If $(q_1,M_1,p_1)$ is a fraction from $K$ to $L$ and 
$(q_2,M_2,p_2)$ is a fraction from $L$ to $O$, then their composition 
$(q_2,M_2,p_2) \fcirc (q_1,M_1,p_1)$ is the fraction $ K\stackrel{q_1 \circ \pr_1}{\leftarrow}M_1 \times_{L} M_2 \stackrel{p_2 \circ \pr_2}{\rightarrow}O$. 

The main property of $\Frac(K,L)$ is that $\pi_0(\Frac(K,L)) \cong \hhom_{\twocatD}(K,L)$, where $\pi_0$ denotes the isomorphism classes of objects. 


\section{Recall on the 3-category of Picard 2-stacks}\label{section:recall}

Let $\ES$ be a site. A \emph{Picard $\ES$-2-stack} $\twoP=(\twoP, \otimes, \ass,\comm)$ is an $\ES$-2-stack in 2-groupoids equipped with a morphism of 2-stacks
$\otimes : \twoP \times \twoP \ra \twoP$, called group law of $\twoP$, and with two natural 2-transformations $
  \ass : \otimes \circ (\otimes \times \id_\twoP) \Ra \otimes \circ (\id_\twoP \times \otimes)$ and 
$\comm : \otimes \circ s \Ra \otimes$ (here $s(X,Y)=(Y,X)$)
expressing the associativity and the commutativity constraints of the group law of $\twoP$, such that $\twoP(U)$ is a Picard 2-category for any object $U$ of $\ES$ (see \cite{MR1301844} for the definition of Picard 2-category).

Let $(\twoP, \otimes_{\twoP}, \ass_{\twoP},\comm_{\twoP})$ and $(\twoQ, \otimes_{\twoQ}, \ass_{\twoQ},\comm_{\twoQ})$ be two Picard 2-stacks. \emph{An additive 2-functor} $(F,\lambda_F): \twoP \ra \twoQ$ is given by a morphism of 2-stacks $F: \twoP \ra \twoQ$ (i.e. a cartesian 2-functor) and a natural 2-transformation $\lambda_F \colon \otimes_{\twoQ} \circ F^2  \Ra F \circ \otimes_{\twoP}$, which are compatible with the natural 2-transformations $\ass_{\twoP},\comm_{\twoP},\ass_{\twoQ},\comm_{\twoQ}$, i.e. which are compatible with the Picard structures carried by the underlying 2-categories $\twoP(U)$ and $\twoQ(U).$

Let $(F,\lambda_F), (G,\lambda_G): \twoP \ra \twoQ$  be additive 2-functors between Picard 2-stacks.
A \emph{morphism of additive 2-functors} $(\theta, \Gamma):(F,\lambda_F) \Ra (G,\lambda_G)$ is given by 
a natural 2-transformation of 2-stacks $\theta: F \Ra G$ and a modification of 2-stacks $\Gamma: \lambda_G \circ  \otimes_{\twoQ}*\theta^2 \Rra \theta*\otimes_{\twoP} \circ \lambda_F $ so that $\theta$ and $\Gamma$ are compatible with the additive structures of $(F,\lambda_F)$ and $(G,\lambda_G)$. 

Let $(\theta_1,\Gamma_1),(\theta_2,\Gamma_2) \colon (F,\lambda_F) \Ra (G,\lambda_G) $ be morphisms of additive 2-functors. A \emph{modification of morphisms of additive 2-functors} $(\theta_1,\Gamma_1) \Rra (\theta_2,\Gamma_2)$ is given by a modification  $\Sigma: \theta_1 \Rra \theta_2$ of 2-stacks such that $(\Sigma*\otimes_{\twoP})\lambda_F \circ  \Gamma_1 = \Gamma_2 \circ \lambda_G(\otimes_{\twoQ}*\Sigma^2).$

Since Picard 2-stacks are fibered in 2-groupoids, morphisms of additive 2-functors are invertible up to modifications of morphisms of additive 2-functors and modifications of morphisms of additive 2-functors are strictly invertible. 

Picard 2-stacks over $\ES$ form a 3-category $2\PICARD$ whose objects are Picard 2-stacks and whose hom-2-groupoid consists of additive 2-functors, morphisms of additive 2-functors, and modifications of morphisms of additive 2-functors.

An \emph{equivalence of Picard 2-stacks} between $\twoP$ and $\twoQ$ is an additive 2-functor $F:\twoP \ra \twoQ$ with $F$ an equivalence of 2-stacks. Two Picard 2-stacks are \emph{equivalent as Picard 2-stacks} if there exists an equivalence of Picard 2-stacks between them.

Any Picard 2-stack admits a global neutral object $e$ and the automorphisms of the neutral object $\Aut(e)$ form a Picard stack.

According to \cite[\S 8]{MR1301844}, for any Picard 2-stack $\twoP$ we define the homotopy groups $\pi_i(\twoP)$ for $i=0,1,2$ as follow
\begin{itemize}
	\item $ \pi_0(\twoP)$ is the sheaffification of the pre-sheaf which associates to each object $U$ of $\ES$ the group of equivalence classes of objects of $\twoP(U)$;
	\item $\pi_1(\twoP) =\pi_0 (\Aut(e))$ with $\pi_0 (\Aut(e))$ the sheaffification of the pre-sheaf which associates to each object $U$ of $\ES$ the group of isomorphism classes of objects of $\Aut(e)(U)$;
	\item $\pi_2(\twoP) =\pi_1 (\Aut(e))$ with  $\pi_1(\Aut(e))$ the sheaf of automorphisms of the neutral object of $\Aut(e)$.
\end{itemize}

The algebraic counter part of Picard 2-stacks are the length 3 complexes of abelian sheaves. In \cite{MR2735751}, the second author associates to a length 3 complex of abelian sheaves $A$ a Picard 2-stack denoted by $2\st(A)$ (see \cite{MR2735751} for the details), getting a 3-functor
$ 2\st: \twocatC \rightarrow  2\PICARD$ from the 3-category $\twocatC$ to the 3-category of Picard 2-stacks.
Although morphisms of length 3 complexes of abelian sheaves induce additive 2-functors between the associated Picard 2-stacks, not all of them are obtained in this way. In this sense, the 1-arrows of $\twocatC$ are not geometric and the reason is their strictness. We resolve this problem by weakening the 3-category $\twocatC$, or in other words by introducing the tricategory $\twocatT$. In \cite{MR2735751}, Tatar shows 

\begin{theorem}\label{thm:tatar}
The 3-functor
\begin{equation}\label{2st}
2\st \colon \xymatrix@1{\twocatT \ar[r] & 2\PICARD,}
\end{equation}
given by sending a length 3 complexes of abelian sheaves $A$ to its associated Picard 2-stacks $2\st(A)=\TORS(\oneA,A^0)$ and a fraction $A \stackrel{q}{\leftarrow}M\stackrel{p}{\rightarrow} B$ to the additive 2-functor $2\st(p)2\st(q)^{-1}: 2\st(A) \ra 2\st(B)$, is a triequivalence.
\end{theorem}

We denote by $[\,\,]$ the inverse triequivalence of $2\st$. From this theorem, one can deduce that

\begin{corollary}\label{tatar:corollary}
The 3-functor $2\st$ induces an equivalence of categories
\begin{equation}\label{2st_flat_flat}
2\st^{\flat\flat} \colon \xymatrix@1{\twocatD \ar[r] & 2\FPICARD,}
\end{equation}
where $2\FPICARD$ is the category of Picard 2-stacks whose objects are Picard 2-stacks and whose arrows are equivalence classes of additive 2-functors.  
\end{corollary}

We denote by $[\,\,]^{\flat\flat}$ the inverse equivalence of $2\st^{\flat\flat}$. The 3-functor $2\st$ and the functor $2\st^{\flat\flat}$ coincide on objects, i.e. if $A$ is a length 3 complex, $2\st(A)=2\st^{\flat\flat}(A)$.

We have the following link between the sheaves $\pi_i$ associated to a Picard 2-stack $\twoP$ and the sheaves $\mathrm{H}^{-i}$ associated to a length 3 complex of abelian sheaves $A$ in degrees -2,-1,0: for $i=0,1,2$
$$\pi_i(\twoP)=\mathrm{H}^{-i}([\twoP]^{\flat\flat}) \qquad \mathrm{and} \qquad \pi_i(2\st(A))=\mathrm{H}^{-i}(A).$$

\begin{example}\label{example}
Let $\twoP$ and $\twoQ$ be two Picard 2-stacks. Denote by $\Hom_{2\PICARD}(\twoP,\twoQ)$
 the Picard 2-stack such that for any object $U$ of $\ES$, the objects of the 2-category $\Hom_{2\PICARD}(\twoP,\twoQ)(U)$ are additive 2-functors from $\twoP(U)$ to $\twoQ(U)$, its 1-arrows are morphisms of additive 2-functors and its 2-arrows are modifications of morphisms of additive 2-functors. By \cite[\S 4]{MR2735751}, in the derived category $\der$ we have the equality
\begin{equation*}
  [\Hom_{2\PICARD}(\twoP,\twoQ)]^{\flat\flat} = \tau_{\leq 0}{\mathrm{R}}{\mathrm{Hom}}\big([\twoP],[\twoQ]\big).
\end{equation*}
With these notation, the hom-2-groupoid of two objects $\twoP$ and $\twoQ$ of the 3-category $2\PICARD$ is just $\Hom_{2\PICARD}(\twoP,\twoQ)$.
\end{example}


\section{Fundamental operations on Picard 2-stacks}

\textit{Product of Picard 2-stacks}

Let $\twoA$ and $\twoB$ be two Picard 2-stacks. 

\begin{definition}\label{def:product}
The \emph{product of $\twoA$ and $\twoB$} is the Picard 2-stack 
$\twoA \times \twoB$ defined as follows: for any object $U$ of $\ES$,
\begin{itemize}
  \item an object of the 2-groupoid $(\twoA \times \twoB)(U)$ is a pair $(X,Y)$ of objects with $X$ an object of $\twoA(U)$ and $Y$ an object of $\twoB(U)$;
  \item a 1-arrow  $(X,Y) \ra (X',Y')$ between two objects of $(\twoA \times \twoB)(U)$ is a pair $(f,g)$ with $f:X \ra  X' $ a 1-arrow of $\twoA(U)$ and $g:Y \ra  Y' $ a 1-arrow of $\twoB(U)$;
  \item a 2-arrow $(f,g) \Ra (f',g')$ between two parallel 1-arrows of $(\twoA \times \twoB)(U)$ is a pair $(\alpha,\beta)$ with $\alpha:f \Ra  f' $ a 2-arrow of $\twoA(U)$ and $\beta:g \Ra  g' $ a 2-arrow of $\twoB(U)$.
\end{itemize}
\end{definition}

\textit{Fibered product of Picard 2-stacks}

Consider now two additive 2-functors $F: \twoA \rightarrow \twoP$ and $G: \twoB \rightarrow \twoP$ between Picard 2-stacks.

\begin{definition}\label{def:fiberedproduct}
The \emph{fibered product} of $\twoA$ and $\twoB$ over $\twoP$ is the Picard 2-stack $\twoA \times_{\twoP} \twoB$ defined as follows: for any object $U$ of $\ES$,
\begin{itemize}
\item an object of the 2-groupoid $(\twoA \times_{\twoP} \twoB)(U)$ is a triple $(X,l,Y)$ where $X$ is an object of $ \twoA(U)$, $Y$ is an object of $ \twoB(U)$ and $l:FX \ra GY$ is a 1-arrow in $\twoP(U)$;
\item a 1-arrow $(X_1,l_1,Y_1) \ra (X_2,l_2,Y_2)$ between two objects of $(\twoA \times_{\twoP} \twoB)(U)$ is given by the triple $(m,\alpha,n)$ where $m:X_1 \ra X_2$ and $n: Y_1 \ra Y_2$ are 1-arrows in $\twoA(U)$ and $\twoB(U)$ respectively, and $\alpha \colon l_2 \circ Fm \Ra Gn \circ l_1$ is a 2-arrow in $\twoP(U)$;
\item a 2-arrow between two parallel 1-arrows $(m,\alpha,n), (m',\alpha',n'):(X_1,l_1,Y_1) \ra (X_2,l_2,Y_2)$ of $(\twoA \times_{\twoP} \twoB)(U)$ is given by the pair $(\theta,\phi)$ where $\theta:m \Ra m'$ and $\phi:n \Ra n'$ are 2-arrows in $\twoA(U)$ and $\twoB(U)$ respectively, satisfying the equation $\alpha' \circ (l_2*F\theta) = (G\phi * l_1) \circ \alpha$ of 2-arrows.
\end{itemize}
\end{definition}

The fibered product $\twoA \times_\twoP \twoB$ is also called the \emph{pull-back $F^*\twoB$ of $\twoB$ via $F:\twoA \ra \twoP$} or the \emph{pull-back $G^*\twoA$ of $\twoA$ via $G:\twoB \ra \twoP$}. It is endowed with two additive 2-functors $\pr_1: \twoA \times_\twoP \twoB \ra \twoA$ and $\pr_2: \twoA \times_\twoP \twoB \ra \twoB$ and a morphism of additive 2-functors $\pi \colon G \circ \pr_2 \Ra F \circ \pr_1$.

The fibered product $\twoA \times_{\twoP} \twoB$  satisfies the following universal property: For every diagram 
\begin{equation*}
\begin{tabular}{c}
\xymatrix{\twoC \ar[r]^{H_1} \ar[d]_{H_2} & \twoA \ar[d]^F\\ \twoB \ar[r]_{G} &\twoP \ultwocell<\omit>{\tau}}
\end{tabular}
\end{equation*}
there exists a 4-tuple $(K,\gamma_1,\gamma_2,\Theta)$, where $K \colon \twoC \ra \twoA \times_{\twoP} \twoB$ is an additive 2-functor, $\gamma_1 \colon \pr_1 \circ K \Ra H_1$ and $\gamma_2 \colon \pr_2 \circ K \Ra H_2$ are two morphisms of additive 2-functors, and $\Theta$ is a modification of morphisms of additive 2-functors
\begin{equation*}
\begin{tabular}{c}
\xymatrix{(G\pr_2)K \ar@2[r]^{\ass} \ar@2[d]_{\pi*K} & G(\pr_2K) \ar@{}[d]|{\kesir{\rotatebox{45}{$\Rra$}}{\Theta}}\ar@2[r]^(0.55){G *\gamma_2}& GH_2 \ar@2[d]^{\tau}\\ (F\pr_1)K \ar@2[r]_{\ass}& F(\pr_1K) \ar@2[r]_(0.55){F * \gamma_1} & FH_1}
\end{tabular}
\end{equation*}
This universal property is unique in the following sense: For any other 4-tuple $(K',\gamma_1',\gamma_2',\Theta')$ as above, there exists a 3-tuple $(\psi,\Sigma_1,\Sigma_2)$, where $\psi \colon K \Ra K'$ is a morphism of additive 2-functors, and $\Sigma_1$, $\Sigma_2$ are two modifications of morphisms of additive 2-functors
\begin{equation*}
\begin{tabular}{cc}
\xymatrix@C=0.3cm@R=0.3cm{\pr_1K \ar@2[rr]^{\pr_1*\psi} \ar@2[ddr]_{\gamma_1} &\ar@{}[dd]|(0.4){\kesir{\Lla}{\Sigma_1}}& \pr_1K' \ar@2[ddl]^{\gamma_1'}\\&&\\&H_1&}
&
\xymatrix@C=0.3cm@R=0.3cm{\pr_2K \ar@2[rr]^{\pr_2*\psi} \ar@2[ddr]_{\gamma_2} &\ar@{}[dd]|(0.4){\kesir{\Lla}{\Sigma_2}}& \pr_2K' \ar@2[ddl]^{\gamma_2' }\\&&\\&H_2&}
\end{tabular}
\end{equation*}
satisfying the compatibility 
\begin{equation}\label{universality_of_pullback_6}
\begin{tabular}{c}
\xymatrix@C=0.58cm{(G\pr_2)K \ar@2{}[dr]|{\cong}\ar@2[d]|{(G\pr_2)*\psi} \ar@2[r]^{\ass} & G(\pr_2K) \ar@2[d]|{G*(\pr_2*\psi)} \ar@2@/^0.75cm/[dr]^{G*\gamma_2}&&&(G\pr_2)K  \ar@2[d]_{\pi*K} \ar@2[r]^{\ass} \ar@2@/_0.75cm/[dl]_{(G\pr_2)*\psi} &G(\pr_2K) \ar@2{}[d]|{\rotatebox{90}{$\Rra$}\Theta}  \ar@2[r]^(0.55){G*\gamma_2} & GH_2 \ar@2[d]^{\tau}\\
(G\pr_2)K' \ar@2[d]_{\pi*K'} \ar@2[r]_{\ass} &G(\pr_2K') \ar@2{}[d]|{\rotatebox{90}{$\Rra$}\Theta'} \ar@2[r]_(0.6){G*\gamma_2'} & GH_2 \ar@2{}[ul]|(0.4){\rotatebox{90}{$\Rra$}G*\Sigma_2} \ar@2[d]^{\tau}\ar@{}[r]|(0.4){=}&(G\pr_2)K' \ar@2@/_0.75cm/[dr]_{\pi*K'} \ar@{}[r]|{\cong}& (F\pr_1)K \ar@{}[dr]|{\cong} \ar@2[r]^{\ass} \ar@2[d]|{(F\pr_1)*\psi}&F(\pr_1K)  \ar@2[r]^(0.55){F*\gamma_1} \ar@2[d]|{F*(\pr_1*\psi)}&FH_1 \ar@{}[dl]|(0.4){{\rotatebox{90}{$\Rra$}F*\Sigma_1}}\\
(F\pr_1)K' \ar@2[r]_{\ass}&F(\pr_1K') \ar@2[r]_(0.6){F*\gamma_1'} &FH_1&&(F\pr_1)K' \ar@2[r]_{\ass} &F(\pr_1K') \ar@2@/_0.75cm/[ur]_{F*\gamma_1'}&}
\end{tabular}
\end{equation}
so that for another 3-tuple $(\psi',\Sigma_1',\Sigma_2')$ as above, there exists a unique modification $\mu:\psi \Rra \psi'$
satisfying the following compatibilities for $i=1,2$
\begin{equation*}
\begin{tabular}{c}
\xymatrix@C=0.3cm@R=0.3cm{&\ar@{}[d]|(0.6){\kesir{\rotatebox{90}{\scriptsize{$\Lla$}}\mu}{}} &&&&&\\ \pr_iK \ar@2@/^0.75cm/[rr]^{\pr_i*\psi} \ar@2[rr]_{\pr_i*\psi'} \ar@2[ddr]_{\gamma_i} &\ar@{}[dd]|{\kesir{\Lla}{\Sigma_i'}}& \pr_iK' \ar@2[ddl]^{\gamma_i'}&&\pr_iK \ar@2[rr]^{\pr_i*\psi} \ar@2[ddr]_{\gamma_i} &\ar@{}[dd]|{\kesir{\Lla}{\Sigma_i}}& \pr_iK' \ar@2[ddl]^{\gamma_i'} \\&&&=&&&\\&H_i &&&&H_i&}
\end{tabular}
\end{equation*}
The cells with $\cong$ in diagram (\ref{universality_of_pullback_6}) commute up to a natural modification due to the Picard structure explained in Example \ref{example}. \\

\textit{Fibered product of length 3 complexes} 

Let $f: A \ra P$ and $g: B \ra P$ be two morphisms of complexes in $\twocatK$. In the category of complexes $\twocatK$, the \emph{naive fibered product} of $A$ and $B$ over $P$ (i.e. the degree by degree fibered product)
 is not the good notion of fibered product of complexes since via the triequivalence of tricategories $2\st$ (\ref{2st}) it doesn't furnish the fibered product of Picard 2-stacks. The good definition of fibered product in $\twocatK$ is the following one:

\begin{definition}\label{def:fibered_product_of_length_3_complexes}
The \emph{fibered product} $A \times_P B$ of $A$ and $B$ over $P$ is the good truncation in degree 0 of the mapping cone of $f-g$ shifted of -1:
\[A \times_P B:=  \tau_{ \leq 0}\big(\MC(f-g)[-1]\big) .\]
\end{definition}

If $A=[A^{-2} \stackrel{\delta_A}{\ra} A^{-1} \stackrel{\lambda_A}{\ra} A^0], B=[B^{-2} \stackrel{\delta_B}{\ra} B^{-1} \stackrel{\lambda_B}{\ra} B^0],P=[P^{-2} \stackrel{\delta_P}{\ra} P^{-1} \stackrel{\lambda_P}{\ra} P^0]$ and $f=(f^{-2},f^{-1},f^0): A \ra P, g=(g^{-2},g^{-1},g^0): B \ra P$, the complex $A \times_P B=\tau_{ \leq 0}\big(\MC(f-g)[-1]\big)$ is explicitly the length 3 complex 
\begin{equation}\label{complex:fibered_product}
\xymatrix@1@C=1.5cm{(A^{-2} + B^{-2}) \oplus 0 \ar[r]^(0.47){\delta_{A \times_P B}} & (A^{-1} + B^{-1}) \oplus P^{-2} \ar[r]^{\lambda_{A \times_P B}} & \ker(0, f^{0}-g^{0}-\lambda_P),}
\end{equation}
where $\ker(0, f^{0}-g^{0}-\lambda_P) \subseteq (A^0 +B^0) \oplus P^{-1}$,  $\delta_{A \times_P B}=
\left( \begin{array}{cc}
\delta_A +\delta_B & 0  \\
f^{-2}-g^{-2} & 0 \end{array} \right),$ 
$\lambda_{A \times_P B}=
 \left( \begin{array}{cc}
\lambda_A +\lambda_B & 0  \\
f^{-1}-g^{-1} & -\delta_P \end{array} \right)$, and the differential from $(A^0 +B^0) \oplus P^{-1}$ to $0 \oplus P^0$ is
$\left( \begin{array}{cc}
0 & 0  \\
f^{0}-g^{0} & -\lambda_P \end{array} \right).$

\begin{proposition}\label{lemma:torsor_fibered_product}
Let $A$, $B$, and $P$ be complexes in $\twocatK$ and let $F: 2\st^{\flat\flat}(A) \ra 2\st^{\flat\flat}(P)$ and $G: 2\st^{\flat\flat}(B) \ra  2\st^{\flat\flat}(P)$ be additive 2-functors induced by the morphisms of complexes $f:A \ra P$ and $g:B \ra P$ in $\twocatK$. Then we have the following equivalence of the Picard 2-stacks 
\[ 2\st^{\flat\flat}(A \times_P B) \cong 2\st^{\flat\flat}(A) \times_{2\st^{\flat\flat}(P)} 2\st^{\flat\flat}(B) .\]
\end{proposition}

\begin{proof} 
To prove this proposition we construct two morphisms
 \[\Theta \colon 2\st^{\flat\flat}(A) \times_{2\st^{\flat\flat}(P)} 2\st^{\flat\flat}(B) \lra 2\st^{\flat\flat}(A \times_P B),\]
 \[\Psi \colon 2\st^{\flat\flat}(A \times_P B) \lra  2\st^{\flat\flat}(A) \times_{2\st^{\flat\flat}(P)} 2\st^{\flat\flat}(B),\]
and show that $\Theta \circ \Psi \cong \Psi \circ \Theta \cong \id$. We first construct $\Theta$: Let $\oneU=(V_{\bullet} \ra U)$ be a hypercover of an object $U$ of $\ES$ (see \cite[\S 2]{Aldrovandi2009687}) and let
$((a,m,\theta),(l,\alpha),(b,n,\phi))$ be a 2-descent datum representing an object of $2\st^{\flat\flat}(A) \times_{2\st^{\flat\flat}(P)} 2\st^{\flat\flat}(B)$ over $U$ relative to $\oneU$ (see \cite[\S 6]{breen-2006}): in particular $(a,m,\theta)$ and $(b,n,\phi)$ are 2-descent data representing objects of $2\st^{\flat\flat}(A)$ and $2\st^{\flat\flat}(B)$ respectively, and  $(l,\alpha) \colon G(b,n,\phi) \ra F(a,m,\theta)$ is a 1-arrow of $2\st^{\flat\flat}(P)$ over $U$ relative to $\oneU$, i.e. $l \in P^{-1}(V_0)$ and $\alpha \in P^{-2}(V_1)$ such that
 \begin{eqnarray}
\nonumber  f^0(a)-g^0(b)&=&\lambda_P(l), \\
\nonumber f^{-1}(m)-g^{-1}(n)&=&\delta_P(\alpha)+d_0^*(l)-d_1^*(l),
\end{eqnarray}
with the property
\begin{equation*}
f^{-2}(\theta)-g^{-2}(\phi)=d_0^*(\alpha)-d_1^*(\alpha)+d_2^*(\alpha).
\end{equation*} 
Confronting the above relations with the complex (\ref{complex:fibered_product}), we deduce that the collection\\
$((a,b,l),(m,n,\alpha),(\theta,\phi))$ is a 2-descent datum representing an object of $2\st^{\flat\flat}(A \times_P B)$ over $U$ relative to $\oneU$. We define $\Theta((a,m,\theta),(l,\alpha),(b,n,\phi))=((a,b,l),(m,n,\alpha),(\theta,\phi)).$ \\
 Now we construct $\Psi$: Let $((a',b',l'),(m',n',\alpha'),(\theta',\phi'))$ be a 2-descent datum representing an object of $2\st^{\flat\flat}(A \times_P B)$ over $U$ relative to $\oneU$. We define its image under $\Psi$ by $((a',m',\theta'),(l',\alpha'),(b',n',\phi'))$. \\
 It follows directly from the definitions of the morphisms $\Psi$ and $\Theta$ that $\Psi \circ \Theta \cong \Theta \circ \Psi \cong \id$.
\end{proof}

We extend the discussion of fibered product of length 3 complexes of abelian sheaves to the tricategory $\twocatT$. Let $f=(q_f, M,p_f)$ be a fraction from $A$ to $P$ and $g=(q_g, N, p_g)$ be a fraction from $B$ to $P$ in $\twocatT$.

\begin{definition}\label{def:fibered_product_of_fractions}
The \emph{fibered product} $A \boxtimes_{P} B$ of $A$ and $B$ over $P$ is the fibered product $M \times_P N$ of $M$ and $N$ over $P$ via  the morphisms of complexes $p_f \colon M \ra P $ and $p_g \colon N \ra P $ in the sense of Definition \ref{def:fibered_product_of_length_3_complexes}:
\[A \boxtimes_{P} B:= M \times_P N =  \tau_{ \leq 0}\big(\MC(p_f-p_g)[-1]\big).\]
\end{definition}

Using Proposition \ref{lemma:torsor_fibered_product} and the fact that $M$ and $N$ are quasi-isomorphic to $A$ and $B$ respectively, we get that the notion of fibered product of complexes in the tricategory $\twocatT$ corresponds to the notion of fibered product of Picard 2-stacks in $2\PICARD$: $2\st(A \boxtimes_P B) \cong 2\st(A) \times_{2\st(P)} 2\st(B).$\\

\textit{Fibered sum of length 3 complexes}

The dual notion of fibered product of complexes is fibered sum. Let  $f: P \ra A$ and $g: P \ra B$ be two morphisms of complexes in $\twocatK$. In the category of complexes $\twocatK$, the \emph{naive fibered sum} of $A$ and $B$ under $P$ (i.e. the fibered sum degree by degree) is not the good notion of fibered sum for complexes. The good definition is 

\begin{definition}\label{def:fibered_sum_of_length_3_complexes}
The \emph{fibered sum} $A +^{P} B$ of $A$ and $B$ under $P$ is the good truncation in degree -2 of the mapping cone of $f-g$:
\[A +^{P} B:=  \tau_{\geq -2}(\MC(f-g)) .\]
\end{definition}

As in the case of fibered products, we extend the definition of fibered sum to complexes in $\twocatT$. Let $f=(q_f, M,p_f)$ be a fraction from $P$ to $A$ and $g=(q_g, N, p_g)$ be a fraction from $P$ to $B$ in $\twocatT$. The complexes
 $A$ and $B$ are not under a common complex, but under the complexes $M$ and $N$ which are quasi-isomorphic to $P$.
So to define the fibered sum of $A$ and $B$ under $P$, we first make the fibered product $M \times_P N$ of $M$ and $N$ over $P$ via the morphisms of complexes $q_f \colon M \ra P $ and $q_g \colon N \ra P $ in the sense of Definition \ref{def:fibered_product_of_length_3_complexes}. We denote by $\pr_M \colon M \times_P N \ra M $ and by $\pr_N \colon M \times_P N \ra N $ the natural projections underlying the fibered product $M \times_P N$. Then, we define

\begin{definition}\label{def:fibered_sum_fractions}
The \emph{fibered sum} $A \boxplus^{P} B$ of $A$ and $B$ under $P$ is the fibered sum  $A +^{M \times_P N} B$ of $A$ and $B$ under $M \times_P N$  via  the morphisms of complexes $p_f \circ \pr_M \colon M \times_P N \ra A $ and $p_g \circ \pr_N \colon M \times_P N \ra B $ in the sense of Definition \ref{def:fibered_sum_of_length_3_complexes}:
\[A \boxplus^{P} B:= A +^{M \times_P N} B =  \tau_{\geq -2}\big(\MC(p_f \circ \pr_M -p_g \circ \pr_N)\big).\] 
\end{definition}

\textit{Fibered sum of Picard 2-stacks}

To define fibered sum of Picard 2-stacks one needs the 2-stackification process. We circumvent this process, which is yet to be defined, by defining the fibered sum of two Picard 2-stacks in $2\PICARD$ as the image, under the triequivalence of tricategories (\ref{2st}), of the fibered sum of the corresponding complexes in $\twocatT$. 

\begin{definition}\label{def:fibered_sum_of_Picard_2_stacks}
The \emph{fibered sum} $\twoA +^{\twoP} \twoB$ of $\twoA$ and $\twoB$ under $\twoP$ is the Picard 2-stack $2\st([\twoA] \boxplus^{[\twoP]} [\twoB])$.
\end{definition}

The fibered sum $\twoA +^{\twoP} \twoB$ is also called the \emph{push-down $F_*\twoB$ of $\twoB$ via $F:\twoP \ra \twoA$} or the \emph{push-down $G_*\twoA$ of $\twoA$ via $G \colon \twoP \ra \twoB$}. It is endowed with two additive 2-functors $\inc_1 \colon \twoA \ra \twoA +^\twoP \twoB$ and $\inc_2 \colon \twoB \ra \twoA +^\twoP \twoB$ and with a morphism of additive 2-functors $\iota \colon  \inc_2 \circ G \Rightarrow \inc_1 \circ F$. Moreover it satisfies the dual universal property of the fibered product. In Appendix \ref{AnnexB} we state this universal property and we sketch the proof of the fact that the fibered sum $\twoA +^{\twoP} \twoB$ satisfies this universal property. \\

\textit{Examples}

Let $\twoA$ and $\twoB$ be two Picard 2-stacks and $F:\twoA \ra \twoB$ be an additive 2-functor. We denote by $\boldsymbol 0$ the Picard 2-stack whose only object is the unit object and whose only 1- and 2-arrows are identities.

\begin{definition}\label{def:homotopy_kernel}
The \emph{homotopy kernel} $\twoKer(F)$ of $F$ is the fibered product $\twoA \times_{\twoB} \boldsymbol 0$ via the additive 2-functor $F:\twoA \ra \twoB$ and the additive 2-functor $\boldsymbol 0 \ra \twoB$. \\
The \emph{homotopy cokernel} $\twoCo(F)$ of $F$ is the fibered sum $ \boldsymbol 0 +^{\twoA} \twoB$ via the additive 2-functor $F:\twoA \ra \twoB$ and the additive 2-functor $\twoA \ra \boldsymbol 0$. 
\end{definition}

Using Proposition \ref{lemma:torsor_fibered_product} we have 
 
\begin{lemma}
\begin{enumerate}
	\item Let $f:A \ra B$ be a morphism of complexes of $\twocatK$ and let $F: 2\st^{\flat\flat}(A) \ra 2\st^{\flat\flat}(B)$ be the additive 2-functor induced by $f$.\\ 
	We have $\twoKer(F) = 2\st^{\flat\flat}\big(\tau_{ \leq 0}(\MC(f)[-1])\big)$ and $\twoCo(F)=2\st^{\flat\flat} \big(\tau_{\geq -2}(\MC(f))\big)$.
 	\item Let $F: \twoA \rightarrow \twoB$ be an additive 2-functor between Picard 2-stacks and let $f=(q_f,M,p_f)$ be the fraction of $\twocatT$ corresponding to $F$ via (\ref{2st}). \\
 	 We have $\twoKer(F)=2\st\big(\tau_{ \leq 0}(\MC(p_f )[-1])\big)$ and $\twoCo(F)=2\st\big(\tau_{\geq -2}(\MC(p_f ))\big)$ .
\end{enumerate}
\end{lemma}


\section{The 3-category of extensions of Picard 2-stacks}\label{section:3_category_of_extensions}

Let $\twoA$ and $\twoB$ be two Picard 2-stacks. 

\begin{definition}\label{def:extpic2stacks}
An \emph{extension} $(I,\twoE,J,\varepsilon)$ of $\twoA$ by $\twoB$ consists of
\begin{itemize}
	\item a Picard 2-stack $\twoE$;
	\item two additive 2-functors $I:\twoB \rightarrow \twoE$ and $J:\twoE \rightarrow \twoA$;
	\item a morphism of additive 2-functors $\varepsilon:J \circ I \Ra 0$ between $J \circ I$ and the null 2-functor $0:\twoB \ra \twoA$,
\end{itemize}
 such that the following equivalent conditions are satisfied:
\begin{itemize}
	\item $\pi_0(J): \pi_0(\twoE) \rightarrow \pi_0(\twoA)$ is surjective and $I$ induces an equivalence of Picard 2-stacks between $\twoB$ and $\twoKer(J)$; 
	\item $\pi_2(I): \pi_2(\twoB) \rightarrow \pi_2(\twoE)$ is injective and $J$ induces an equivalence of  Picard 2-stacks between $\twoCo(I)$ and $\twoA$.
\end{itemize}
\end{definition}

Let $(I,\twoE,J,\varepsilon)$ be an extension of $\twoA$ by $\twoB$ and let $(K,\twoF,L,\varsigma)$ be an extension of $\twoC$ by $\twoD$.

\begin{definition}\label{def:morextpic2stacks}
A \emph{morphism of extensions} $(I,\twoE,J,\varepsilon) \ra (K,\twoF,L,\varsigma)$ is given by the collection $(H,F,G,\alpha,\beta,\Phi)$ where
\begin{itemize}
\item $H:\twoB \ra \twoD$, $F:\twoE \ra \twoF$, and $G:\twoA \ra \twoC$ are additive 2-functors;
\item $\alpha:F \circ I \Ra K \circ H$ and $\beta: L \circ F \Ra G \circ J$ are morphisms of additive 2-functors;
\item $\Phi$ is the modification of morphisms of additive 2-functors,
\begin{equation}\label{morphism_of_extensions}
\begin{tabular}{c}
\xymatrix@R=1.25cm{ (LF)I \ar@2[r]^{\ass} \ar@2[d]_{\beta*I} &L(FI) \ar@2[r]^{L * \alpha} \ar@{}[dr]|{\kesir{\rotatebox{45}{$\Rra$}}{\Phi}}& L(KH) \ar@2[r]^(0.55){\invass}& (LK)H \ar@2[d]^{\varsigma*H}\\(GJ)I \ar@2[r]_{\ass} & G(JI) \ar@2[r]_{G * \varepsilon} & G0 \ar@2[r]_{\mu_G}& 0H}
\end{tabular}
\end{equation}
where $\mu_G \colon G \circ 0 \Ra 0 \circ H$ is the morphism of additive 2-functors defined as follows: For any $U \in \ES$ and for any object $X$ of $\twoB(U)$, the component of $\mu_G$ at $X$ is the natural arrow  $[\mu_G]_X \colon G e_{\twoA} \ra e_{\twoC}$ in $\twoC(U)$.
\end{itemize}
\end{definition}

\begin{definition}\label{def:equivalenceofext}
Two extensions $\twoE_1=(I_1,\twoE_1,J_1,\varepsilon_1)$ and $\twoE_2=(I_2,\twoE_2,J_2,\varepsilon_2)$ of $\twoA$ by $\twoB$ are \emph{equivalent as extensions of $\twoA$ by $\twoB$} if there exists a morphism of extensions from $\twoE_1$ to $\twoE_2$ inducing identities on $\twoA$ and on $\twoB$.
\end{definition}

In other words, $\twoE_1$ and $\twoE_2$ are equivalent as extensions of $\twoA$ by $\twoB$ if it exists an additive 2-functor $F:\twoE_1 \ra \twoE_2$, two morphisms of additive 2-functors $\alpha \colon F \circ I_1 \Ra I_2 \circ \id_\twoB$ and $\beta \colon J_2 \circ F \Ra \id_{\twoA} \circ J_1$ and a modification of morphisms of additive 2-functors $\Phi$ such that $(\id_\twoB,F,\id_\twoA,\alpha,\beta,\Phi)$ is a morphism of extensions.

Let $(H_1, F_1, G_1, \alpha_1,\beta_1,\Phi_1)$ and $(H_2, F_2, G_2, \alpha_2,\beta_2,\Phi_2)$ be two morphisms of extensions 
from $(I,\twoE,J,\varepsilon)$ to $(K,\twoF,L,\varsigma)$

\begin{definition}\label{def:2morextpic2stacks}
A \emph{2-morphism of extensions} $(H_1, F_1, G_1, \alpha_1,\beta_1,\Phi_1) \Ra (H_2, F_2, G_2, \alpha_2,\beta_2,\Phi_2)$
 is given by the collection $(\gamma, \delta, \epsilon, \Psi, \Omega)$ where 
\begin{itemize}
\item $\gamma:H_1 \Ra H_2$, $\delta:F_1 \Ra F_2$, $\epsilon:G_1 \Ra G_2$ are morphisms of additive 2-functors;
\item  $\Psi$ and $\Omega$ are modifications of morphisms of additive 2-functors
\begin{equation}\label{diagram:1-morphism_modification}
\begin{tabular}{cc}
\xymatrix{F_1I \ar@2[r]^{\alpha_1} \ar@2[d]_{\delta * I} & KH_1 \ar@2[d]^{K * \gamma}\\F_2I \ar@2[r]_{\alpha_2}& KH_2 \ar@{}[ul]|{\kesir{\rotatebox{45}{$\Rra$}}{\Psi}}}
&
\xymatrix{LF_1 \ar@2[r]^{\beta_1} \ar@2[d]_{L * \delta} & G_1J \ar@2[d]^{\epsilon*J}\\LF_2 \ar@2[r]_{\beta_2}& G_2J \ar@{}[ul]|{\kesir{\rotatebox{45}{$\Rra$}}{\Omega}}}
\end{tabular}
\end{equation}
\end{itemize}
so that the pasting of the 3-arrows in the diagram
\begin{equation}\label{diagram:compatibility_condition_of_2_morphisms_of_extensions}
\begin{tabular}{l}
\xymatrix{(LF_2)I \ar@{}[dr]|{\kesir{\rotatebox{270}{$\Rra$}}{\Omega*I}} \ar@2[r]^{\beta_2*I} & (G_2J)I \ar@{}[dr]|{\cong} \ar@2[r]^{\ass} & G_2(JI) \ar@{}[dr]|{\cong} \ar@2[r]^{G_2*\varepsilon} & G_20 \ar@2[dr]^{\mu_{G_2}} &\\
(LF_1)I \ar@2[r]_{\beta_1*I} \ar@2[u]^{L*(\delta*I)}\ar@2[d]_{\ass}& (G_1J)I \ar@{}[dr]|{\kesir{\rotatebox{270}{$\Rra$}}{\Phi_1}} \ar@2[r]_{\ass} \ar@2[u]|{(\epsilon*J)*I}&G_1(JI)\ar@2[r]_{G_1*\varepsilon} \ar@2[u]|{\epsilon*(JI)}& G_10 \ar@{}[r]|{\cong}\ar@2[d]_{\mu_{G_1}} \ar@2[u]^{\epsilon*0}& 0H_2\\
L(F_1I) \ar@2[r]_{L*\alpha_1} & L(KH_1) \ar@2[r]_{\invass} & (LK)H_1 \ar@2[r]_{\varsigma*H_1} & 0H_1 \ar@2[ur]_{0*\delta}&}
\end{tabular}
\end{equation}
is equal to the pastings of the 3-arrows in the diagram
\begin{equation}\label{diagram:compatibility_condition_of_2_morphisms_of_extensions_bis}
\begin{tabular}{l}
\xymatrix{(LF_2)I \ar@2[r]^{\beta_2*I} \ar@2[dr]^{\ass}& (G_2J)I \ar@2[r]^{\ass} & G_2(JI) \ar@{}[d]|{\kesir{\rotatebox{270}{$\Rra$}}{\Phi_2}}\ar@2[r]^{G_2*\varepsilon} & G_20 \ar@2[dr]^{\mu_{G_2}} &\\
(LF_1)I \ar@{}[r]|{\cong} \ar@2[u]^{(L*\delta)*I} \ar@2[d]_{\ass} & L(F_2I)  \ar@{}[d]|{\kesir{\rotatebox{270}{$\Rra$}}{L*\Psi}} \ar@2[r]^{L*\alpha_2} &L(KH_2) \ar@{}[d]|{\cong} \ar@2[r]^{\invass}& (LK)H_2  \ar@{}[d]|{\cong} \ar@2[r]^{\varsigma*H_2}& 0H_2\\
L(F_1I) \ar@2[r]_{L*\alpha_1} \ar@2[ur]|{L*(\delta*I)}& L(KH_1) \ar@2[r]_{\invass} \ar@2[ur]|{L*(K*\gamma)}& (LK)H_1 \ar@2[r]_{\varsigma*H_1} \ar@2[ur]|{(LK)*\gamma}& 0H_1 \ar@2[ur]_{0*\gamma}&}
\end{tabular}
\end{equation}
\end{definition} 
In the diagrams above the symbol $\cong$ inside a cell means that the cell commutes up to a natural modification of morphisms of additive 2-functors explained in Example \ref{example}.

Let $(\gamma, \delta, \epsilon, \Psi, \Omega)$ and $(\gamma',\delta', \epsilon',\Psi', \Omega')$ be two 2-morphisms of extensions from $(H_1, F_1, G_1,\alpha_1,\beta_1,\Phi_1)$ to $(H_2,F_2,G_2, \alpha_2,\beta_2,\Phi_2)$.

\begin{definition}\label{def:3morextpic2stacks}
A \emph{3-morphism of extensions} 
$(\gamma, \delta, \epsilon, \Psi, \Omega)  \Rra (\gamma',\delta', \epsilon',\Psi', \Omega')$
is given by three modifications of morphisms of additive 2-functors $\Gamma: \gamma \Rra \gamma'$, $\Delta:\delta \Rra \delta'$, and $\Upsilon: \epsilon \Rra \epsilon'$ satisfying the equation
\begin{equation}\label{diagram:compatibility_condition_of_3_morphisms_of_extensions}
\begin{tabular}{c}
\xymatrix{&F_1I \ar@2[r]^{\alpha_1} \ar@{}[dr]|{\kesir{\rotatebox{45}{$\Rra$}}{\Psi'}}\ar@2[d]^{\scriptsize{\rotatebox{90}{$\delta' *I$}}}_{\kesir{\Rra}{\Delta*I}} \ar@2@/_0.75cm/[d]_{\scriptsize{\rotatebox{90}{$\delta *I$}}} & KH_1 \ar@2[d]^{K * \gamma'} \ar@{}[dr]|{=}&F_1I \ar@{}[dr]|{\kesir{\rotatebox{45}{$\Rra$}}{\Psi}} \ar@2[r]^{\alpha_1} \ar@2[d]_{\delta * I} & KH_1 \ar@2[d]_{\scriptsize{\rotatebox{90}{$K * \gamma$}}} \ar@2@/^0.75cm/[d]^{\scriptsize{\rotatebox{90}{$K * \gamma'$}}}_{\kesir{\Rra}{K*\Gamma}}&\\ &F_2I \ar@2[r]_{\alpha_2}&KH_2&F_2I \ar@2[r]_{\alpha_2}&KH_2&&&}
\end{tabular}
\end{equation}
and a similar equation between the modifications $\Omega$, $\Omega'$, $\Delta$, and $\Upsilon$.
\end{definition}

\begin{def-prop}\label{def-prop:2-groupoid_of_extensions}
Let $\twoE=(I,\twoE,J,\varepsilon)$ be an extension of $\twoA$ by $\twoB$ and let $\twoF=(K,\twoF,L,\varsigma)$ be an extension of $\twoC$ by $\twoD$. Then the $\ES$-2-stack $\HomExt(\twoE,\twoF)$ whose
\begin{itemize}
	\item objects are morphisms of extensions from $\twoE$ to $\twoF$;
	\item 1-arrows are 2-morphisms of extensions;
	\item 2-arrows are 3-morphisms of extensions;
\end{itemize}
is a 2-groupoid, called the 2-groupoid of morphisms of extensions from $\twoE$ to $\twoF$. 
\end{def-prop}

\begin{proof} The proof is left to the reader.
\end{proof}

The extensions of Picard 2-stacks over $\ES$ form a 3-category $\Ext_{2\PICARD}$ where objects are extensions of Picard 2-stacks and where the hom-2-groupoid of two extensions $\twoE$ and $\twoF$ is $\HomExt(\twoE,\twoF)$. For any two Picard 2-stacks $\twoA$ and $\twoB$, we denote by $\Ext(\twoA,\twoB)$ the 3-category of extensions of $\twoA$ by $\twoB$.

\section{Extensions of length 3 complexes via fractions}
Let $A$ and $B$ be complexes of $\twocatT$.

\begin{definition}\label{def:extension_via_fractions}
An \emph{extension} $E=(i,E,j,R)$ of $A$ by $B$ in the tricategory $\twocatT$ consists of 
\begin{itemize}
	\item a complex $E$ of $\twocatT$;
	\item two fractions $i=(q_i,M,p_i)$ from $B$ to $E$ and $j=(q_j,N,p_j)$ from $E$ to $A$ of $\twocatT$;
	\item a 1-arrow of fractions $R=(r,R,r'):j \fcirc i \Ra 0$ between $j \fcirc i$ and the trivial fraction 0;
\end{itemize}
such that the following equivalent conditions are satisfied:
\begin{enumerate}[(a)]
\item ${\mathrm{H}}^{0}(p_j) \circ ({\mathrm{H}}^{0}(q_j))^{-1}: {\mathrm{H}}^{0}(E) \ra {\mathrm{H}}^{0}(A)$ is surjective and $i$ induces a quasi-isomorphism between $B$ and $ \tau_{\leq 0} (\MC(p_j)[-1])$;
\item  ${\mathrm{H}}^{-2}(p_i)\circ ({\mathrm{H}}^{-2}(q_i))^{-1} \colon {\mathrm{H}}^{-2}(B) \ra {\mathrm{H}}^{-2}(E)$ is injective and $j$ induces a quasi-isomorphism between $ \tau_{\geq -2}(\MC(p_i))$ and $A$.
\end{enumerate}
\end{definition}

Let $(i,E,j,R)$ be an extension of $A$ by $B$ and $(k,F,l,S)$ be extension of $C$ by $D$ with $i=(q_i,M,p_i)$, $j=(q_j,N,p_j)$, $k=(q_k,K,p_k)$, and $l=(q_l,L,p_l)$.

\begin{definition}\label{def:morphism_of_extensions_via_fractions}
A \emph{morphism of extensions} $(i,E,j,R) \ra (k,F,l,S)$ is given by the collection $(f,g,h,T,U,\omega)$ where
\begin{itemize}
	\item $f=(q_f,Q_f,p_f) \colon E \ra F$, $g=(q_g,Q_g,p_g) \colon A \ra C$, and $h=(q_h,Q_h,p_h) \colon B \ra D$ are fractions; 
	\item $T=(t,T,t'):f \fcirc i \Ra k \fcirc h$ and $U=(u,U,u') \colon l \fcirc f \Ra g \fcirc j$ are 1-arrows of fractions
\begin{equation*}\label{diagram:morphisms_of_extensions_2_bis}
\begin{tabular}{c}
		\xymatrix@C=0.6cm@R=0.4cm{B \ddrrtwocell<\omit>{\mspace{10mu}T}& M \ar[r]^{p_i} \ar[l]_{q_i} & E & N \ar[r]^{p_j} \ar[l]_{q_j} & A\\
				Q_h \ar[u]^{q_h}\ar[d] _{p_h}&& Q_f \ar[u]_{q_f} \ar[d]^{p_f} && Q_g \ar[u]_{q_g}\ar[d]^{p_g} \\
				D & K \ar[r]_{p_k} \ar[l]^{q_k} & F & L \ar[r]_{p_l} \ar[l]^{q_l}& C \uulltwocell<\omit>{\mspace{-10mu} U}}
\end{tabular}
\end{equation*}
	\item $\omega$ is a 2-arrow of fractions from the pasting of the 1-arrows of fractions $(s,S,s')$, $(u,U,u')$, and $(t,T,t')$ to the 1-arrow of fraction $(r,R,r')$
\begin{equation}\label{morphism_of_extensions_fractions}
\begin{tabular}{c}
\xymatrix@R=1.25cm{ l(fi) \ar@2[r]^{\ass} \ar@2[d]_{l * T} &(lf)i \ar@2[r]^{U*i} \ar@{}[dr]|{\kesir{\rotatebox{45}{$\Rra$}}{\omega}}& (gj)i \ar@2[r]^{\invass}& g(ji) \ar@2[d]^{g * R}\\l(kh) \ar@2[r]_{\ass} & (lk)h \ar@2[r]_{S*h} & 0_Dh \ar@2[r]_{\mu_h}& g0_B}
\end{tabular}
\end{equation}
where $0_B=(\id_B,B,0) \colon B \ra A$, $0_D=(\id_D,D,0) \colon D \ra C$, and $\mu_h$ is the 1-arrow of fractions given by triple $(\id_{Q_h},Q_h,q_h)$. 
\end{itemize}
\end{definition}

Let  $(i,E,j,R)$ be an extension of $A$ by $B$ and $(k,F,l,S)$ be an extension of $C$ by $D$. Let $(f_1,g_1,h_1,T_1,U_1,\omega_1)$ and $(f_2,g_2,h_2,T_2,U_2,\omega_2)$ be two morphisms of extensions from $(i,E,j,R)$ to $(k,F,l,S)$.

\begin{definition}\label{def:2-morphism_of_extensions_via_fractions}
A \emph{2-morphism of extensions} 
$(f_1,g_1,h_1,T_1,U_1,\omega_1) \Ra (f_2,g_2,h_2,T_2,U_2,\omega_2)$ 
is given by the collection $(X_f,X_g,X_h,\sigma,\tau)$ where
\begin{itemize}
\item $X_f=(x_f,X_f,x'_f): f_1 \Rightarrow f_2$, $X_g=(x_g,X_g,x'_g) : g_1 \Rightarrow g_2$, and $X_h=(x_h,X_h,x'_h): h_1 \Rightarrow h_2 $ are 1-arrows of fractions;
\item $\sigma$ and $\tau$ are 2-arrows of fractions
\begin{equation*}
\begin{tabular}{cc}
\xymatrix{f_1i \ar@2[r]^{T_1} \ar@2[d]_{X_f * i} & kh_1 \ar@2[d]^{k * X_h}\\f_2i \ar@2[r]_{T_2}& kh_2 \ar@{}[ul]|{\kesir{\rotatebox{45}{$\Rra$}}{\sigma}}}
&
\xymatrix{lf_1 \ar@2[r]^{U_1} \ar@2[d]_{l * X_f} & g_1j \ar@2[d]^{X_g*j}\\lf_2 \ar@2[r]_{U_2}& g_2j \ar@{}[ul]|{\kesir{\rotatebox{45}{$\Rra$}}{\tau}}}
\end{tabular}
\end{equation*}
 such that $\sigma$, $\tau$, $\omega_1$, and $\omega_2$ satisfy a compatibility condition which can be obtained from diagrams analog to (\ref{diagram:compatibility_condition_of_2_morphisms_of_extensions}) and (\ref{diagram:compatibility_condition_of_2_morphisms_of_extensions_bis}).
 \end{itemize}
\end{definition}

Let $(X_f,X_g,X_h,\sigma,\tau)$ and $(Y_f,Y_g,Y_h,\mu,\nu)$ be two 2-morphisms of extensions.

\begin{definition}\label{def:3-morphism_of_extensions_via_fractions}
A \emph{3-morphism of extensions}
$(X_f,X_g,X_h,\sigma,\tau) \Rra (Y_f,Y_g,Y_h,\mu,\nu)$
is given by three 2-arrows of fractions $\alpha:X_f \Rra Y_f$, $\beta:X_g \Rra Y_g$, and $\gamma:X_h \Rra Y_h$ (i.e. isomorphisms) such that all regions in the following diagrams commute
\begin{equation*}\label{diagram:3-morphisms_of_extensions_via_fractions_1}
\begin{tabular}{ccc}
\xymatrix@C=0.001cm@R=5pt{&&&&& Q_{f_1} \ar[ddrrrrr]^(0.6){p_{f_1}} \ar[ddlllll]_(0.6){q_{f_1}}&&&&&\\
          &&&&&&\\
          E &&&& X_f \ar@{-->}[rr]^{\alpha} \ar[uur] \ar[ddr] && Y_f \ar[uul] \ar[ddl] &&&& F\\
          &&&&&&&&\\
          &&&&&Q_{f_2} \ar[uurrrrr]_(0.6){p_{f_2}} \ar[uulllll]^(0.6){q_{f_2}}&&&&&}
          &
\xymatrix@C=0.001cm@R=5pt{&&&&& Q_{g_1} \ar[ddrrrrr]^(0.6){p_{g_1}} \ar[ddlllll]_(0.6){q_{g_1}}&&&&&\\
          &&&&&&\\
          A &&&& X_g \ar@{-->}[rr]^{\beta} \ar[uur] \ar[ddr] && Y_g \ar[uul] \ar[ddl] &&&& C\\
          &&&&&&&&\\
          &&&&&Q_{g_2} \ar[uurrrrr]_(0.6){p_{g_2}} \ar[uulllll]^(0.6){q_{g_2}}&&&&&}
          &
\xymatrix@C=0.001cm@R=5pt{&&&&& Q_{h_1} \ar[ddrrrrr]^(0.6){p_{h_1}} \ar[ddlllll]_(0.6){q_{h_1}}&&&&&\\
          &&&&&&\\
          B &&&& X_h \ar@{-->}[rr]^{\gamma} \ar[uur] \ar[ddr] && Y_h \ar[uul] \ar[ddl] &&&& D\\
          &&&&&&&&\\
          &&&&&Q_{h_2} \ar[uurrrrr]_(0.6){p_{h_2}} \ar[uulllll]^(0.6){q_{h_2}}&&&&&}
\end{tabular}
\end{equation*}
and such that $\alpha$, $\beta$, $\gamma$, $\sigma$, $\tau$, $\mu$, $\nu$ satisfy the compatible conditions which are given by a commutative diagram of 3-arrows analog to (\ref{diagram:compatibility_condition_of_3_morphisms_of_extensions}).
\end{definition}

As for extensions of Picard 2-stacks we have the following Proposition whose proof is left to the reader:

\begin{def-prop}\label{def-prop:bigroupoid_of_extensions_via_fractions}
Let $E=(i,E,j,R)$ be an extension of $A$ by $B$ and $F=(k,F,l,S)$ be an extension of $C$ by $D$. Then the $\ES$-2-stack  $\HomExt(E,F)$ whose
\begin{itemize}
	\item objects are morphisms of extensions from $E$ to $F$;
	\item 1-arrows are 2-morphisms of extensions;
	\item 2-arrows are 3-morphisms of extensions;
\end{itemize}
is a bigroupoid, called the bigroupoid of morphisms of extensions from $E$ to $F$.
\end{def-prop}

The extensions of length 3 complexes in $\twocatT$ form a tricategory $\Ext_{\twocatT}$ where
 objects are extensions of length 3 complexes in $\twocatT$ and
 where the hom-bigroupoid of two extensions $E$ and $F$ is $\HomExt(E,F)$.
For any two length 3 complexes $A$ and $B$ of $\twocatT$, we denote by $\Ext(A,B)$ the tricategory of extensions of $A$ by $B$.

\begin{remark}\label{rem:exact-ext}
Let $E=(i,E,j,R)$ be an extension of $A$ by $B$ with $i=(q_i,M,p_i)$ and $j=(q_j,N,p_j)$. The morphism of complexes $p_j \colon N \ra A$ can be completed into a distinguished triangle $\MC(p_j)[-1]  \ra N \stackrel{p_j} \ra A \ra +$ which is isomorphic to $B  \stackrel{i } \ra E \stackrel{j } \ra A \ra +$ in $\der$. Similarly, the morphism of complexes $p_i \colon M \ra E$ can be completed into a distinguished triangle $M \stackrel{p_i} \ra E \ra \MC(p_i) \ra +$ which is isomorphic to $B  \stackrel{i } \ra E \stackrel{j} \ra A \ra +$ in $\der$.
\end{remark} 

As an immediate consequence of the above Definitions we have

\begin{proposition}\label{proposition:triequivalence_of_extensions}
The triequivalence $2\st$ induces a triequivalence between $\Ext_{2\PICARD}$ and $\Ext_{\twocatT}$. 
\end{proposition}


\section{Operations on extensions of Picard 2-stacks}\label{Baer_sum_of_extensions}

Let $\twoE=(I,\twoE,J,\varepsilon)$ be an extension of the Picard 2-stack $\twoA$ by the Picard 2-stack $\twoB$ and let $G:\twoA' \ra \twoA$ be an additive 2-functor. Recall that we denote by $\boldsymbol 0$ the Picard 2-stack whose only object is the unit object and whose only 1- and 2-arrows are identities.

\begin{definition} The \emph{pull-back $G^*\twoE$ of the extension $\twoE$ via the additive 2-functor $G:\twoA' \ra \twoA$} is the fibered product $\twoE \times_\twoA \twoA'$ of $\twoE$ and $\twoA'$ over $\twoA$ via $J:\twoE \ra \twoA$ and $G:\twoA' \ra \twoA$.
\end{definition}

\begin{lemma} \label{lemma:pull-back}
The pull-back $G^*\twoE $ of $\twoE$ via $G:\twoA' \ra \twoA$ is an extension of $\twoA'$ by $\twoB.$
\end{lemma}

\begin{proof} 
Let $G^*\twoE=(G^*\twoE, \pr_1, \pr_2, \pi_G)$ be the pull-back of $\twoE$ via $G$ and $J$, with $\pr_1 \colon G^*\twoE \ra \twoA'$ and $\pr_2 \colon G^*\twoE \ra \twoE$ the underlying additive 2-functors and $\pi_G \colon J \circ \pr_2 \Ra G \circ \pr_1$ the underlying morphism of additive 2-functor. From the morphism of additive 2-functors $\varepsilon: J \circ I \Ra 0$ we get the morphism of additive 2-functors
\begin{equation*}
\begin{tabular}{c}
\xymatrix{\twoB \ar[r]^{0} \ar[d]_{I}  &  \twoA' \ar[d]^{G}\\ \twoE \ar[r]_{J} & \twoA \ultwocell<\omit>{\mspace{-10 mu}\varepsilon_G}}
\end{tabular}
\end{equation*}
Therefore according to the universal property of the pull-back, there exists a 4-tuple $(I',\gamma_1,\gamma_2,\Theta)$ consisting of an additive 2-functor $I': \twoB \ra G^*\twoE$, two morphisms of additive 2-functors $\gamma_1 \colon \pr_1 \circ I' \Ra 0$ and $\gamma_2 \colon \pr_2 \circ I' \Ra I$, and a modification of morphisms of additive 2-functors  $\Theta$
\begin{equation*}\label{diagram:pull_back_extension_1}
\begin{tabular}{c}
\xymatrix{(J\pr_2)I' \ar@2[r]^{\ass} \ar@2[d]_{\pi_G*I'} & J(\pr_2I') \ar@{}[d]|{\kesir{\rotatebox{45}{$\Rra$}}{\Theta}}\ar@2[r]^(0.55){J *\gamma_2}& JI \ar@2[d]^{\varepsilon_G}\\ (G\pr_1)I' \ar@2[r]_{\ass}& G(\pr_1I') \ar@2[r]_(0.55){G * \gamma_1} & G0}
\end{tabular}
\end{equation*}
Moreover by composing the equivalence of Picard 2-stacks $\twoB \cong \twoKer(J)= \twoE \times_\twoA \boldsymbol 0$ with the natural equivalence of Picard 2-stacks $ \twoE \times_\twoA \boldsymbol 0 \cong \twoE \times_{\twoA}  \twoA'  \times_{\twoA'} \boldsymbol 0 = \twoKer(\pr_1)$, we get that $\twoB$ is equivalent to the Picard 2-stack $\twoKer(\pr_1)$. Finally the surjectivity of $\pi_0(J): \pi_0(\twoE) \ra \pi_0(\twoA)$ implies the surjectivity of  $\pi_0(\pr_1): \pi_0(G^*\twoE) \ra \pi_0(\twoA') $. Hence $(I',G^*\twoE,\pr_1,\gamma_1) $ is an extension of $\twoA'$ by $\twoB$.
\end{proof}

The dual notion of pull-back of an extension is the push-down of an extension. Let $\twoE=(I,\twoE,J,\varepsilon)$ be an extension of $\twoA$ by $\twoB$ and let $F:\twoB \rightarrow \twoB'$ be an additive 2-functor.

\begin{definition} The \emph{push-down $F_*\twoE$ of the extension $\twoE$ via the additive 2-functor $F:\twoB \rightarrow \twoB'$} is the fibered sum $\twoE +^\twoB \twoB'$ of $\twoE$ and $\twoB'$ under $\twoB$ via $F:\twoB \rightarrow \twoB'$ and $I: \twoB \ra \twoE$.
\end{definition}

Dualizing the proofs done for the pull-back of an extension, we get that the push-down $F_*\twoE $ of the extension $\twoE$ via $F:\twoB \rightarrow \twoB'$ is an extension of $\twoA$ by $\twoB'$ which is endowed with a universal property.

Now we can define the group law for extensions of $\twoA$ by $\twoB$ using pull-back and push-down of extensions.
Let $\twoE$ and $\twoE'$ be two extensions of $\twoA$ by $\twoB$. Remark that $\twoE \times \twoE'$ is an extension of $\twoA \times \twoA$ by $\twoB \times \twoB.$

\begin{definition} \label{def:sum} 
 The \emph{sum $\twoE + \twoE'$ of the extensions $\twoE$ and $\twoE'$} is the following extension of $\twoA$ by $\twoB$
\begin{equation}\label{eq:+}
D_{\twoA}^* (\otimes_\twoB)_* ( \twoE \times \twoE'),
\end{equation}
where $D_{\twoA}: \twoA \ra \twoA \times \twoA $ is the diagonal additive 2-functor of $\twoA$ and $\otimes_{\twoB}: \twoB \times \twoB \ra \twoB$ is the morphism of 2-stacks underlying the Picard 2-stack $\twoB$ (i.e. $\otimes_{\twoB}$ is the group law of $\twoB$).
\end{definition}

\begin{proposition}\label{proposition:sumofext}
The sum given in Definition \ref{def:sum} equipes the set $\Ext^1(\twoA,\twoB)$ of equivalence classes of extensions of $\twoA$ by $\twoB$ with an abelian group law, where the neutral element is the equivalence class of the extension $\twoA \times \twoB$, and the inverse of an equivalence class $\twoE$ is the equivalence class of $-\twoE= (-\id_{\twoB})_*\twoE$.
\end{proposition}

\begin{proof} 
 Associativity: Following the definition of the sum and using the universality of pull-back and push-down, we obtain
\begin{center}
\begin{tabular}{ccl}
$(\twoE_1 +\twoE_2) + \twoE_3$ &$=$& $D_{\twoA}^*(\otimes_{\twoB})_*\big[\big(D_{\twoA}^*(\otimes_{\twoB})_*(\twoE_1 \times \twoE_2)\big) \times \twoE_3\big]$\\ 
&$\cong$ & $D_{\twoA}^*(\otimes_{\twoB})_*\big[(D_{\twoA} \times \id_{\twoA})^*\big((\otimes_{\twoB})_*(\twoE_1 \times \twoE_2) \times \twoE_3\big)\big]$\\
& $\cong$ & $D_{\twoA}^*(\otimes_{\twoB})_*\big[(D_{\twoA} \times \id_{\twoA})^*\big[(\otimes_{\twoB} \times \id_{\twoB})_*((\twoE_1 \times \twoE_2) \times \twoE_3)\big]\big]$\\
&$\cong$ & $D_{\twoA}^*(D_{\twoA} \times \id_{\twoA})^*(\otimes_{\twoB})_*(\otimes_{\twoB} \times \id_{\twoB})_*\big[(\twoE_1 \times \twoE_2) \times \twoE_3\big]$\\
&$\cong$ & $\big[(D_{\twoA} \times \id_{\twoA}) \circ D_{\twoA}\big]^*\big(\otimes_{\twoB} \circ (\otimes_{\twoB} \times \id_{\twoB})\big)_*\big((\twoE_1 \times \twoE_2) \times \twoE_3\big)$
\end{tabular}
\end{center}
By repeating the above arguments starting with $\twoE_1 +(\twoE_2 + \twoE_3)$, we find that  $\twoE_1 +(\twoE_2 + \twoE_3) \cong [(\id_{\twoA} \times D_{\twoA}) \circ D_{\twoA}]^*\big(\otimes_{\twoB} \circ (\id_{\twoB} \times \otimes_{\twoB})\big)_*((\twoE_1 \times \twoE_2) \times \twoE_3)$. Using the associativity constraint $\ass \colon \otimes_{\twoB} \circ (\otimes_{\twoB} \times \id_{\twoB}) \Ra \otimes_{\twoB} \circ (\id_{\twoB} \times \otimes_{\twoB})$ of a Picard 2-stacks and observing that $(D_{\twoA} \times \id_{\twoA})\circ D_{\twoA} = (\id_{\twoA} \times D_{\twoA}) \circ D_{\twoA}$, we find that $(\twoE_1 +\twoE_2) + \twoE_3 \cong \twoE_1 +(\twoE_2 + \twoE_3)$.   
 
Commutativity: It is clear from the formula (\ref{eq:+}).

Neutral element: It is the product $\twoA \times \twoB$ of the extension $\twoA=(\boldsymbol 0 \rightarrow \twoA,\twoA,  \id:\twoA \rightarrow \twoA,0)$ of $\twoA$ by $\boldsymbol 0$ with the extension 
	$\twoB=(\id:\twoB \rightarrow \twoB, \twoB,\twoB \rightarrow \boldsymbol 0,0)$ of $\boldsymbol 0$ by $\twoB$. 
\end{proof}


\section{Homological interpretation of extensions of Picard 2-stacks}

Let $\twoA$ and $\twoB$ be two Picard 2-stacks.

\begin{lemma}\label{lemma:the_equivalence}
Let $\twoE=(I,\twoE,J,\varepsilon)$ be an extension of $\twoA$ by $\twoB$. Then the Picard 2-stack $\Hom_{2\PICARD}(\twoA,\twoB)$ is equivalent to $\HomExt(\twoE,\twoE)$. In particular, $\HomExt(\twoE,\twoE)$ is endowed with a Picard 2-stack structure.
\end{lemma}
\begin{proof}
The equivalence is given via the additive 2-functor 
\begin{eqnarray}
\nonumber \Hom_{2\PICARD}(\twoA,\twoB) &\longrightarrow & \HomExt(\twoE,\twoE) \\
 \nonumber  F& \mapsto & \big(a \mapsto a + IFJ (a) \big).
\end{eqnarray}
\end{proof}

By the above Lemma, the homotopy groups $\pi_i(\HomExt(\twoE,\twoE))$ for $i=0,1,2$ are abelian groups.
Since by definition $\Ext^{-i}(\twoA,\twoB)=\pi_{i}(\HomExt(\twoE,\twoE))$, we have

\begin{corollary} 
The sets $\Ext^i(\twoA,\twoB)$, for $i=0,-1,-2$, are abelian groups.
\end{corollary}

\begin{proof}[Proof of Theorem \ref{thm:introext} for i=0,-1,-2]
According to Lemma \ref{lemma:the_equivalence}, the homotopy groups of\\
 $\Hom_{2\PICARD}(\twoA,\twoB)$ and $\HomExt(\twoE,\twoE)$ are isomorphic and so by Example \ref{example} we conclude that 
$\Ext^i(\twoA,\twoB) \cong \pi_{-i}(\Hom_{2\PICARD}(\twoA,\twoB)) \simeq {\mathrm{H}}^i \big(\tau_{\leq 0}{\mathrm{R}}\hhom([\twoA],[\twoB])\big) = \hhom_{\der}([\twoA],[\twoB][i])$.
\end{proof}

Before we prove Theorem \ref{thm:introext} for $i=1$, we state the following Definition:

\begin{definition}\label{definition:split_extension}
An extension $\twoE=(I,\twoE,J,\varepsilon)$ of $\twoA$ by $\twoB$ is \emph{split} if one of the following equivalent conditions is satisfied:
\begin{enumerate}
\item there exists an additive 2-functor $V: \twoE \ra \twoB$ and a morphism of additive 2-functors $\alpha: V \circ I \Rightarrow \id_{\twoB}$;
\item there exists an additive 2-functor $U: \twoA \ra \twoE$ and a morphism of additive 2-functors $\beta: J \circ U \Rightarrow \id_{\twoA}$;
\item $\twoE$ is equivalent as extension of $\twoA$ by $\twoB$ (see Definition \ref{def:equivalenceofext}) to the neutral object $\twoA \times \twoB$ of the group law defined in (\ref{eq:+}). 
\end{enumerate}
\end{definition}

\begin{proof}[Proof of Theorem \ref{thm:introext} for i=1]
First we construct a morphism from the group $\Ext^1(\twoA,\twoB)$ of equivalence classes of extensions of $\twoA$ by $\twoB$ to the group $\eext^1([\twoA],[\twoB])$
\[ \Theta \colon \Ext^{1}(\twoA,\twoB) \longrightarrow \eext^1([\twoA],[\twoB]),\]
and a morphism from the group $\eext^1([\twoA],[\twoB])$ to the group $\Ext^1(\twoA,\twoB)$
\[ \Psi \colon {\mathrm{Ext}}^1([\twoA],[\twoB]) \longrightarrow \Ext^{1}(\twoA,\twoB).\]
Then we check that $\Theta \circ \Psi = \id = \Psi \circ \Theta $ and that $\Theta$ is an homomorphism of groups.

Before the proof we fix the following notation: if $A$ is a complex of $\twocatD$ we denote by $\twoA$ the corresponding Picard 2-stack $2\st(A)$. Moreover if $f:A \ra E$ is a morphism in $\twocatD$, we denote by $F: \twoA \ra \twoE$ a representative of the equivalence class of additive 2-functors corresponding to the morphism $f$ via the equivalence of categories (\ref{2st_flat_flat}). 
 
(1) Construction of $\Theta$: Consider an extension $\twoE=(I,\twoE,J,\varepsilon)$ of $\twoA$ by $\twoB$ and denote by $E=(i,E,j,R)$ the corresponding extension of $A=[\twoA]$ by $B=[\twoB]$ in the tricategory $\twocatT$. By Remark \ref{rem:exact-ext} to the extension $E$ is associated the distinguished triangle $B \stackrel{i} \ra E \stackrel{j}{\ra} A \ra +$ in $\der$ which furnishes the long exact sequence
\begin{equation}\label{eqn:thm12}
 \xymatrix@1@C=0.6cm{\cdots \ar[r] & \hhom_{\der}(A,B)  \ar[r]^{i \circ} & \hhom_{\der}(A,E) \ar[r]^{j \circ }& \hhom_{\der}(A,A) \ar[r]^(0.55){\partial} & \eext^1(A,B) \ar[r] & \cdots}
\end{equation}
We set $\Theta(\twoE)= \partial (\id_A).$ The naturality of the connecting map $\partial$ implies that $\Theta(\twoE)$ depends only on the equivalence class of the extension $\twoE$.

\begin{lemma} 
If $\eext^1(A,B)=0$, then every extension of $\twoA$ by $\twoB$ is split.
\end{lemma}

\begin{proof} 
By the long exact sequence (\ref{eqn:thm12}), if the cohomology group $\eext^1(A,B)$ is zero, the identity morphisms $\id_A \colon A \rightarrow A$ lifts to a morphism $f \colon A \rightarrow E$ in $\twocatD$ which furnishes an additive 2-functor $F \colon \twoA \rightarrow \twoE$ such that $J \circ F \cong \id_{\twoA}$. Hence, $\twoE$ is a split extension of $\twoA$ by $\twoB$.
\end{proof}

The above lemma means that $\Theta(\twoE)$ is an obstruction for the extension $\twoE$  to be split: $\twoE$ is split if and only if $\id_A \colon A \rightarrow A$ lifts to $\hhom_{\der}(A,E)$, if and only if $\Theta(\twoE)$ vanishes in $\eext^1(A,B)$.

(2) Construction of $\Psi$: Choose a complex $K=[K^{-2} \rightarrow  K^{-1} \rightarrow K^0]$ of $\twocatD$ such that $K^{-2}$, $K^{-1},K^0$ are injective and such that there exists an injective morphism of complexes $s \colon B \ra K$. We complete $s$ into a distinguished triangle
\begin{equation}\label{eqn:thm2.1}
\xymatrix@1{B \ar[r]^s & K \ar[r]^(0.35)t & \MC(s) \ar[r] & +,}
\end{equation}
in $\der$. Setting $L= \tau_{\geq -2} \MC(s)$, the above distinguished triangle furnishes an extension of Picard 2-stacks 
\begin{equation*}\label{eqn:thm2.3}
\xymatrix@1{\twoB \ar[r]^S & \twoK \ar[r]^T & \twoL ,}
\end{equation*}
and the long exact sequence
\begin{equation}\label{eqn:thm3}
 \xymatrix@1@C=0.6cm{\cdots \ar[r] & \hhom_{\der}(A,B)  \ar[r] & \hhom_{\der}(A,K) \ar[r]^{ t \circ }& \hhom_{\der}(A,L) \ar[r]^(0.55){\partial} & \eext^1(A,B) \ar[r] &0.}
\end{equation}
Given an element $x$ of $\eext^1(A,B)$, choose an element $u$ of $\hhom_{\der}(A,L)$ such that $\partial (u) =x$. 
The pull-back $U^*\twoK$ of the extension $\twoK$ via the additive 2-functor $U \colon \twoA \ra \twoL$ corresponding to the morphism $u \colon A \ra L$ of $\der$ is an extension of $\twoA$ by $\twoB$ by Lemma \ref{lemma:pull-back}.\\
We set $\Psi(x) =  U^*\twoK $ i.e. to be precise $\Psi(x)$ is the equivalence class of the extension $U^*\twoK$ of $\twoA$ by $\twoB$. Now we check that the morphism $\Psi$ is well defined, i.e. $\Psi(u)$ doesn't depend on the lift $u$ of $x$. Let $u' \in \hhom_{\der}(A,L)$ be another lift of $x$. From the exactness of the sequence (\ref{eqn:thm3}), there exists $f \in \hhom_{\der}(A,K)$ such that $u'-u=t \circ f ,$
i.e. we have the following commutative diagram 
\begin{equation*}\label{eqn:thm5.1}
\begin{tabular}{c}
\xymatrix{\twoA \ar[r]^{\id_{\twoA}} \ar[d]_F & \twoA \ar[d]^{U'-U}\\ \twoK \ar[r]_{T} & \twoL \ultwocell<\omit>{\pi}}
\end{tabular}
\end{equation*}
Consider now the pull-back $(U'-U)^*\twoK$ of the extension $\twoK$ via the additive 2-functor $U'-U: \twoA \ra \twoL$. The universal property of the pull-back $(U'-U)^*\twoK$ applied to the above diagram furnishes an additive 2-functor $H:\twoA \ra (U'-U)^*\twoK$ and a morphism of additive 2-functors  $\alpha: \pr_1 \circ H \Rightarrow \id_{\twoA}$ (here $\pr_1: (U'-U)^*\twoK \ra \twoA$ is the additive 2-functor underlying the extension $(U'-U)^*\twoK$ of $\twoA$ by $\twoB$). Hence from Definition \ref{definition:split_extension} the extension $(U'-U)^*\twoA$ is split, which means that the extensions $U'^*\twoA$ and $U^*\twoA$ are equivalent.

(3) $ \Theta \circ \Psi = \id$: With the notation of (2), given an element $x$ of $\eext^1(A,B)$, choose an element $u$ of $\hhom_{\der}(A,L)$ such that $\partial (u) =x$. By definition $\Psi(x)=  U^*\twoK$. Because of the naturality of the connecting map $\partial$, the following diagram commutes
\begin{equation*}\label{eqn:thm6}
\begin{tabular}{c}
\xymatrix@C=0.3cm@R=0.4cm{\hhom_{\der}(A,B) \ar[rr] \ar[dd]_{\id} && \hhom_{\der}(A,[U^*\twoK]) \ar[rr] \ar[dd] && \hhom_{\der}(A,A) \ar[rr]^{\partial} \ar[dd]^{u \circ} && \eext^1(A,B) \ar[dd]^{\id}\\&&&&&&\\ \hhom_{\der}(A,B) \ar[rr] && \hhom_{\der}(A,K) \ar[rr] && \hhom_{\der}(A,L) \ar[rr]_{\partial} && \eext^1(A,B)}
\end{tabular}
\end{equation*}
Therefore $\Theta(\Psi(x))=\Theta(U^*\twoK)=\partial(\id_A)=\partial(u \circ \id_A)=\partial(u)=x$, i.e. $\Theta$ surjective.

(4) $\Psi \circ \Theta= \id$: Consider an extension $\twoE=(I,\twoE,J,\varepsilon)$ of $\twoA$ by $\twoB$ and denote by $E=(i,E,j,R)$ the corresponding extension of $A=[\twoA]$ by $B=[\twoB]$ in $\twocatT$. As in (2), choose a complex $K$ of $\twocatD$ such that $K^{-2}$, $K^{-1},K^0$ are injective and such that there exists an injective morphism of complexes $s \colon B \ra K$. Complete $s \colon B \ra K$ into the distinguished triangle (\ref{eqn:thm2.1}) and let $L= \tau_{\geq -2} \MC(s)$. The injectivity of $K$ furnishes a lift $u \colon E \ra K$ of the morphism of complexes $s \colon B \ra K$. From the axioms of the triangulated categories, there exists a morphism $v' \colon A \ra L$ giving rise to the morphism of distinguished triangles
\begin{equation}\label{eqn:thm7}
\begin{tabular}{c}
\xymatrix@C=0.3cm@R=0.3cm{B \ar[rr]^i \ar[dd]_{\id} && E \ar[rr]^j \ar[dd]^u && A \ar[rr] \ar[dd]^{v'} && B[1] \ar[dd]^{\id}\\&&&&&&\\ B \ar[rr]_s && K \ar[rr]_t && L \ar[rr] && B[1]}
\end{tabular}
\end{equation}
which leads to a morphism of long exact sequences
\begin{equation}\label{eqn:thm8}
\begin{tabular}{c}
\xymatrix@C=0.3cm@R=0.3cm{\hhom_{\der}(A,B) \ar[rr]^{i \circ} \ar[dd]_{\id} && \hhom_{\der}(A,E) \ar[rr]^{j \circ} \ar[dd]^{u \circ} && \hhom_{\der}(A,A) \ar[rr]^{\partial} \ar[dd]^{v' \circ} && \eext^1(A,B) \ar[dd]^{\id}\\&&&&&&\\ \hhom_{\der}(A,B) \ar[rr]_{s \circ} && \hhom_{\der}(A,K) \ar[rr]_{t \circ} && \hhom_{\der}(A,L) \ar[rr]_{\partial} && \eext^1(A,B)}
\end{tabular}
\end{equation}
Let $\Theta(\twoE)=\partial(\id_A)=y$ with $y$ an element of $\eext^1(A,B)$. By definition $\Psi(y)=V^*\twoK$ with $v$ an element of $ \hhom_{\der}(A,L)$ such that $\partial(v)=y$. From the commutativity of the diagram (\ref{eqn:thm8}), $v'-v= t \circ f$ with $f \in \hhom_{\der}(A,K),$ which shows as in (2) that the extensions $V^*\twoK$ and $ V'^*\twoK$ are equivalent. From the universal property of the pull-back $V'^*\twoK$ applied to the central square of (\ref{eqn:thm7}), there exists an additive 2-functor $H: \twoE \ra V'^*\twoK$ and two morphisms of additive 2-functors $ pr_1 \circ H \Rightarrow J$, $pr_2 \circ H \Rightarrow U$ (here  $\pr_1: V'^*\twoK \ra \twoA$ and $\pr_2: V'^*\twoK \ra \twoK$ are the additive 2-functors underlying the pull-back $V'^*\twoK$), which furnish a morphism of extensions $(\id_\twoB,H,\id_\twoA,\alpha: H \circ I \Rightarrow I' \circ id_\twoB,\beta: pr_1 \circ H \Rightarrow id_\twoA \circ J,\Phi)$ from $\twoE$ to $V'^*\twoK$ inducing the identity on $\twoA$ and $\twoB$ (here $I': \twoB \rightarrow V'^*\twoK$ is the additive two functor underlying the extension $V'^*\twoK$ of $\twoA$ by $\twoB$). By definition, the extensions $\twoE$ and $V'^*\twoK$ are then equivalent. Summarizing $\Psi (\Theta(\twoE))=\Psi (y)= V^*\twoK \cong V'^*\twoK \cong \twoE$, i.e. $\Theta$ is injective.

(5) $\Theta$ is a group homomorphism: Consider two extensions $\twoE, \twoE'$ of $\twoA$ by $\twoB$. With the notations of (2) we can suppose that $\twoE= U^*\twoK$ and $\twoE'= U'^*\twoK$ with $U,U':\twoA \ra \twoL$ two additive 2-functors corresponding to two morphisms $u,u':A \ra L$ of $\twocatD$. Now by definition of sum in $\Ext^1(\twoA,\twoB)$ (see formula (\ref{eq:+})), we have
\begin{eqnarray}
\nonumber \twoE + \twoE' & = & D_{\twoA}^* (\otimes_{\twoB})_*  \big(U^*\twoK \times U'^*\twoK \big)\\
\nonumber  & = & D_{\twoA}^* (\otimes_{\twoB})_* (U \times U')^* (\twoK \times \twoK)\\
\nonumber  & = & (U + U')^* D_{\twoL}^* (\otimes_{\twoB})_*  (\twoK \times \twoK)\\
\nonumber  & = & (U + U')^*  (\twoK + \twoK)
\end{eqnarray}
where $D_{\twoA}:\twoA \rightarrow \twoA \times \twoA $ and $D_{\twoL}:\twoL \rightarrow \twoL \times \twoL$ are the diagonal additive 2-functors of $\twoA$ and $\twoL$ respectively,  and $\otimes_{\twoB}: \twoB \times \twoB \rightarrow \twoB$ is the morphism of 2-stacks underlying the Picard 2-stack $\twoB$. This calculation shows that $\Theta ( \twoE + \twoE') = \partial (u+u')$ where $\partial \colon \hhom_{\der}(A,L) \ra \eext^1(A,B)$ is the connecting map of the long exact sequence (\ref{eqn:thm3}). Hence, $ \Theta ( \twoE + \twoE')=\partial (u+u') = \partial (u)+\partial (u')= \Theta ( \twoE )+ \Theta( \twoE')$ . 
\end{proof}

\appendix

\section{Long exact sequence involving homotopy groups}\label{Annex}

\begin{proposition}
Let $(I,\twoE,J,\varepsilon)$ be an extension of $\twoA$ by $\twoB$. There exists two connecting morphisms $\Gamma: \pi_2(\twoA) \ra \pi_1(\twoB)$ and $\Delta:\pi_1(\twoA) \ra \pi_0(\twoB)$ such that the sequence of abelian sheaves
\begin{equation}\label{long_exact_sequence}
\begin{tikzpicture}[baseline=(current bounding box.center), descr/.style={fill=white,inner sep=2.5pt},scale=2.25]
\node (O) at (-3,0) {$0$};
\node (A) at (-2,0) {$\pi_2(\twoB)$};
\node (B) at (-1,0) {$\pi_2(\twoE)$};
\node (C) at (0,0) {$\pi_2(\twoA)$};
\node (D) at (1,0) {$\pi_1(\twoB)$};
\node (E) at (2,0) {$\pi_1(\twoE)$};
\node (F) at (-2,-1) {$\pi_1(\twoA)$};
\node (G) at (-1,-1) {$\pi_0(\twoB)$};
\node (H) at (0,-1) {$\pi_0(\twoE)$};
\node (I) at (1,-1) {$\pi_0(\twoA)$};
\node (J) at (2,-1) {$0$};
\node (L) at (-1,-0.45) {$$};
\node (M) at (1,-0.45) {$$};
\path[->, font=\scriptsize]
(O) edge (A)
(A) edge node[above]{$\pi_2 (I)$} (B)
(B) edge node[above]{$\pi_2 (J)$} (C)
(C) edge node[above]{$\Gamma$} (D)
(D) edge node[above]{$\pi_1 (I)$} (E)
(F) edge node[below]{$\Delta$} (G)
(G) edge node[below]{$\pi_0 (I)$} (H)
(H) edge node[below]{$\pi_0 (J)$} (I)
(I) edge (J)
(L) edge[-] node[descr]{$\pi_1 (J)$} (M);
\draw [rounded corners]
(2.1,-0.1) -- (2.1,-0.45) -- (0.936,-0.45);
\draw [->, rounded corners]
(-0.936,-0.45) -- (-1.9,-0.45) -- (-1.9,-0.87);
\end{tikzpicture}
\end{equation}
is a long exact sequence.
\end{proposition}

\begin{proof}
Consider the additive 2-functor $\Lambda : \Aut(e_{\twoA})\rightarrow \twoKer(J)$
defined as follows: Any $\varphi \in \Aut(e_{\twoA})(U)$ is sent to $(e_{\twoE},\varphi \circ \mu_J)$ with $\mu_J \colon J(e_{\twoE}) \ra e_{\twoA},$ and any 1-arrow $\beta \colon \varphi \Ra \psi$ in $\Aut(e_{\twoA})(U)$ is sent to $(\id_{e_{\twoE}},\beta*\mu_J)$. On the classifying sheaves, this additive 2-functor induces the morphisms $\Lambda_0: \pi_0(\Aut(e_{\twoA})) \ra \pi_0(\twoKer(J))$ and $\Lambda_1:\pi_1(\Aut(e_{\twoA})) \ra \pi_1(\twoKer(J))$. Recalling that $\twoB \cong \twoKer(J)$, we define $\Gamma$ and $\Delta$ as
\begin{align}
\nonumber	\Delta=\Lambda_0 & \colon \pi_1(\twoA) = \pi_0(\Aut(e_{\twoA})) \lra \pi_0({\twoKer(J)})=\pi_0(\twoB), \\
\nonumber	\Gamma=\Lambda_1 & \colon \pi_2(\twoA) = \pi_1(\Aut(e_{\twoA})) \lra \pi_1({\twoKer(J)})=\pi_1(\twoB) .
\end{align}

From the extension $(I,\twoE,J,\varepsilon)$, we obtain a sequence of Picard stacks
\begin{equation}\label{complex:long_exact_sequence_2}
\xymatrix@1@C=1.5cm{\Aut(e_{\twoB}) \ar[r]^{I_{\oneA}} & \Aut(e_{\twoE}) \ar[r]^{J_{\oneA}} &\Aut(e_{\twoA}),}
\end{equation}
where $I_{\oneA}$ and $J_{\oneA}$ are additive functors defined as follows: for any $U \in S$ and $\varphi:e_{\twoB} \ra e_{\twoB}$ in $\Aut(e_{\twoB})(U)$, $I_{\oneA}(\varphi)$ is an automorphism of $e_{\twoE}$ over $U$ such that $\alpha_{I}: \mu_I \circ I(\varphi)\Rightarrow I_{\oneA}(\varphi) \circ \mu_I$, where $\mu_I$ and $\alpha_I$ are respectively a 1-arrow and a 2-arrow of $\twoE(U)$ which result from the additivity of $I$. Similarly for $J_{\twoA}$. The sequence (\ref{complex:long_exact_sequence_2}) is a complex of Picard stacks with $\varepsilon_{\oneA} \colon J_{\oneA} \circ I_{\oneA} \Ra 0$ obtained from $\varepsilon \colon J \circ I \Ra 0$. It exists a functor
$
\tilde{I}_{\oneA} :\Aut(e_{\twoB}) \rightarrow \Ker(J_{\oneA})
$
that sends $\varphi \in \Aut(e_{\twoB})(U)$ to $(I_{\oneA}(\varphi),\varepsilon_{\oneA}(\varphi)) \in \Ker(J_{\oneA})(U)$. Since $\twoB \cong \twoKer(J)$, $\Aut(e_{\twoB}) \cong \Aut(e_{\twoKer(J)})$. Moreover, $\Aut(e_{\twoKer(J)}) \cong \Ker(J_{\oneA})$ where $e_{\twoKer(J)}=(e_{\twoE},\mu_J)$. Therefore the functor $\tilde{I}_{\oneA}$ is an equivalence and so the sequence of Picard stacks (\ref{complex:long_exact_sequence_2}) is left exact. Applying to it \cite[Proposition 6.2.6]{Aldrovandi2009687}, we get the exactness of (\ref{long_exact_sequence}) from the left end to $\pi_1(\twoA)$. 

The exactness at $\pi_0(\twoE)$ follows from the equivalence $\twoKer(J) \cong \twoB$. The surjectivity of $\pi_0(J)$ follows from the definition of extension. It remains to show the exactness at $\pi_0(\twoB)$ and $\pi_1(\twoA)$.

Exactness at $\pi_0(\twoB)$: Let $[(X,\varphi)] \in \pi_0(\twoKer(J))(U)=\pi_0(\twoB)$ so that $\pi_0(I)[(X,\varphi)] =[e_{\twoE}],$ i.e. there exists a 1-arrow $\psi \colon e_{\twoE} \ra X$ in $\twoE(U)$. Consider the class of the automorphism $\chi \in \pi_0(\Aut(e_{\twoA}))(U)=\pi_1(\twoA)(U)$ such that $\beta : \varphi \circ J(\psi) \Rightarrow \chi \circ \mu_J$ with $\beta$ a 2-arrow. Then $\Delta[\chi]= \Lambda_0[\chi]=[(e_{\twoE},\chi \circ \mu_J)]=[(X,\varphi)]$.

Exactness at $\pi_1(\twoA)$: Let $[\varphi] \in \pi_0(\Aut(e_{\twoA}))(U)=\pi_1(\twoA)(U)$ such that $\Delta[\varphi]=\Lambda_0[\varphi]=[\Lambda\varphi]=[(e_{\twoE},\mu_J)]$. That is, $(e_{\twoE},\varphi \circ \mu_J) \cong (e_{\twoE},\mu_J)$. Then there exists $\psi \colon e_{\twoE} \ra e_{\twoE}$ in $\Aut(e_{\twoE})(U)$ and $\beta:\mu_J \circ J(\psi) \Ra \varphi \circ \mu_J$ in $\twoA(U)$. Thus $\pi_1(J)[\psi]=[\varphi]$. 
\end{proof}

\section{Universal property of the fibered sum}\label{AnnexB}

\begin{proposition}
 The fibered sum $\twoA +^{\twoP} \twoB$ of $\twoA$ and $\twoB$ under $\twoP$ defined in Definition \ref{def:fibered_sum_of_Picard_2_stacks} satisfies the following universal property: For every diagram 
\begin{equation}\label{push_down_diagram_1}
\begin{tabular}{c}
\xymatrix{\twoP \ar[r]^{F} \ar[d]_{G} & \twoA \ar[d]^{H_1}\\ \twoB \ar[r]_{H_2} &\twoC \ultwocell<\omit>{\tau}}
\end{tabular}
\end{equation}
there exists a 4-tuple $(K,\gamma_1,\gamma_2,\Theta)$, where $K \colon \twoA +^{\twoP} \twoB \ra \twoC$ is an additive 2-functor, $\gamma_1 \colon K \circ \inc_1 \Ra H_1$ and $\gamma_2 \colon K \circ \inc_2 \Ra H_2$ are two morphisms of additive 2-functors, and $\Theta$ is a modification of morphisms of additive 2-functors
\begin{equation}\label{universality_of_pullback_4}
\begin{tabular}{c}
\xymatrix{K(\inc_2G) \ar@2[r]^{\invass} \ar@2[d]_{K*\iota} & (K\inc_2)G \ar@{}[d]|{\kesir{\rotatebox{45}{$\Rra$}}{\Theta}}\ar@2[r]^(0.55){\gamma_2*G}& H_2G \ar@2[d]^{\tau}\\ K(\inc_1F) \ar@2[r]_{\invass}& (K\inc_1)F \ar@2[r]_(0.55){\gamma_1*F} & H_1F}
\end{tabular}
\end{equation}
This universal property is unique in the following sense: For any other 4-tuple $(K',\gamma_1',\gamma_2',\Theta')$ as above, there exists a 3-tuple $(\psi,\Sigma_1,\Sigma_2)$, where $\psi \colon K \Ra K'$ is a morphism of additive 2-functors, and $\Sigma_1$, $\Sigma_2$ are two modifications of morphisms of additive 2-functors
\begin{equation*}
\begin{tabular}{cc}
\xymatrix@C=0.3cm@R=0.3cm{K\inc_1 \ar@2[rr]^{\psi*\inc_1} \ar@2[ddr]_{\gamma_1} &\ar@{}[dd]|(0.4){\kesir{\Lla}{\Sigma_1}}& K'\inc_1 \ar@2[ddl]^{\gamma_1'}\\&&\\&H_1&}
&
\xymatrix@C=0.3cm@R=0.3cm{K\inc_2 \ar@2[rr]^{\psi*\inc_2} \ar@2[ddr]_{\gamma_2} &\ar@{}[dd]|(0.4){\kesir{\Lla}{\Sigma_2}}& K'\inc_2 \ar@2[ddl]^{\gamma_2' }\\&&\\&H_2&}
\end{tabular}
\end{equation*}
satisfying the compatibility dual to (\ref{universality_of_pullback_6}) so that for another 3-tuple $(\psi',\Sigma_1',\Sigma_2')$ as above, there exists a unique modification $\mu:\psi \Rra \psi'$ satisfying the following compatibilities for $i=1,2$
\begin{equation*}
\begin{tabular}{c}
\xymatrix@C=0.3cm@R=0.3cm{&\ar@{}[d]|(0.6){\kesir{\rotatebox{90}{\scriptsize{$\Lla$}}\mu}{}} &&&&&\\ K\inc_i \ar@2@/^0.75cm/[rr]^{\psi*\inc_i} \ar@2[rr]_{\psi'*\inc_i} \ar@2[ddr]_{\gamma_i} &\ar@{}[dd]|{\kesir{\Lla}{\Sigma_i'}}& K'\inc_i \ar@2[ddl]^{\gamma_i'}&&K\inc_i \ar@2[rr]^{\psi*\inc_i} \ar@2[ddr]_{\gamma_i} &\ar@{}[dd]|{\kesir{\Lla}{\Sigma_i}}& K'\inc_i \ar@2[ddl]^{\gamma_i'} \\&&&=&&&\\&H_i &&&&H_i&}
\end{tabular}
\end{equation*}
\end{proposition}

\begin{proof} For simplicity, we assume that all additive 2-functors correspond to morphisms of complexes.
If $A=[A^{-2} \stackrel{\delta_A}{\ra} A^{-1} \stackrel{\lambda_A}{\ra} A^0], B=[B^{-2} \stackrel{\delta_B}{\ra} B^{-1} \stackrel{\lambda_B}{\ra} B^0],P=[P^{-2} \stackrel{\delta_P}{\ra} P^{-1} \stackrel{\lambda_P}{\ra} P^0]$ and $f=(f^{-2},f^{-1},f^0): P \ra A, g=(g^{-2},g^{-1},g^0): P \ra B$, the fibered sum $A +^P B=\tau_{\geq -2}(\MC(f-g))$ of $A$ and $B$ under $P$ is explicitly the length 3 complex 
\begin{equation}\label{complex:fibered_sum}
\xymatrix@1@C=1.5cm{\coker(\delta_P,f^{-2}-g^{-2}) \ar[r]^(0.53){\delta_{A +^P B}} & P^0 \oplus (A^{-1} + B^{-1}) \ar[r]^(0.53){\lambda_{A +^P B}} & 0\oplus(A^0+B^0),}
\end{equation}
where $\coker(\delta_P,f^{-2}-g^{-2})$ is the quotient of the abelian sheaf $P^{-1} \oplus (A^{-2} +B^{-2})$ by the image of the morphism $\left( \begin{array}{cc}
\delta_P & 0  \\
f^{-2}-g^{-2} & 0 \end{array} \right),$  $\delta_{A +^P B}=
\left( \begin{array}{cc}
\lambda_P & 0  \\
f^{-1}-g^{-1} & -\delta_A-\delta_B \end{array} \right),$ 
$\lambda_{A +^P B}=
 \left( \begin{array}{cc}
0 & 0  \\
f^{0}-g^{0} & -\lambda_A-\lambda_B \end{array} \right)$.
  Let $\oneU=(V_{\bullet} \ra U)$ be a hypercover of an object $U$ of $\ES$. According to the complex (\ref{complex:fibered_sum}), a 2-descent datum representing an object of $\twoA+^{\twoP}\twoB$ over $U$ relative to $\oneU$ is the collection $((a,b),(p,k,l),(m,\alpha,\beta))$, where $(a,b) \in (A+B)(V_0)$, $(p,k,l) \in (P^{0}\oplus(A^{-1}+B^{-1}))(V_1)$, and $(m,\alpha,\beta) \in (P^{-1}\oplus(A^{-2}+B^{-2}))(V_2)$ satisfy the relations 
\begin{eqnarray}
\nonumber  f^0(p)-\lambda_A(k)& = & d_0^*(a)-d_1^*(a), \\
\nonumber -g^0(p)-\lambda_B(l)& = & d_0^*(b)-d_1^*(b),\\
\nonumber \lambda_P(m)& = & d_0^*(p)-d_1^*(p)+d_2^*(p),\\
\nonumber f^{-1}(m)-\delta_A(\alpha)& = &d_0^*(k)-d_1^*(k)+d_2^*(k), \\
\nonumber -g^{-1}(m)-\delta_B(\beta)& = & d_0^*(l)-d_1^*(l)+d_2^*(l),
\end{eqnarray}
so that there exists $\rho \in P^{-2}(V_3)$ with the property
\begin{eqnarray}
\nonumber d_0^*(m)-d_1^*(m)+d_2^*(m)-d_3^*(m)& = &\delta_P(\rho),\\
\nonumber d_0^*(\alpha)-d_1^*(\alpha)+d_2^*(\alpha)-d_3^*(\alpha)& = &f^{-2}(\rho),\\
\nonumber d_0^*(\beta)-d_1^*(\beta)+d_2^*(\beta)-d_3^*(\beta)& = &-g^{-2}(\rho).
\end{eqnarray}
From these relations we deduce that the collections 
\[(\lambda_A(k),f^{-1}(m),f^{-2}(\rho)) \textrm{ and } (\lambda_B(l),-g^{-1}(m),-g^{-2}(\rho)),\] 
are actually 2-descent data representing objects of $\twoA$ and $\twoB$ over $V_0$ relative to $\oneU$, respectively.

Construction of $K$: $K$ takes the 2-descent datum $((a,b),(p,k,l),(m,\alpha,\beta))$ to the collection
\[(h_1^0(\lambda_A(k))+h_2^0(\lambda_B(l)),h_1^{-1}(f^{-1}(m))-h_2^{-1}(g^{-1}(m)),h_1^{-2}(f^{-2}(\rho))-h_2^{-2}(g^{-2}(\rho))),\]
where $h_1^i$ and $h_2^i$ are the components at the degree $i$ of the morphisms of complexes that correspond to $H_1$ and $H_2$ respectively.

Construction of $\gamma_1$ and $\gamma_2$: Let $(a,k,\alpha)$ be a 2-descent datum representing an object of $\twoA$ over $U$ relative to $\oneU$. Remark that $\inc_1(a,k,\alpha)= ((-a,0),(0,k,0),(0,-\alpha,0))$ and $K \circ \inc_1(a,k,\alpha)=(h_1^0(\lambda_A(k)),0,0)$ is 2-descent datum representing an object of $\twoC$ over $V_0$ relative to $\oneU$. On the other hand, the image of $(a,k,\alpha)$ under the morphisms $H_1$ is the 2-descent datum $(h_1^0(a),h_1^{-1}(k),h_1^{-2}(\alpha))$ representing an object over $U$ relative to $\oneU$ whose pullback to $V_0$ is the collection $(h_1^0(\lambda_A(k)),h_1^{-1}(\delta_A(\alpha)),0)$. Then the component of $\gamma_1$ at $(a,k,\alpha)$ is the 1-arrow $(0,h_1^{-2}(\alpha))$. Similarly, the component of $\gamma_2$ at $(b,l,\beta)$ is $(0,h_2^{-2}(\beta))$.

Construction of $\Theta$: Let $(p,m,\rho)$ be a 2-descent datum representing an object of $\twoP$ over $U$ relative to $\oneU$.  The component of the 2-arrow obtained by composing the 2-arrows $\gamma_2*G$ and $\tau$ on the top and the right faces of the diagram (\ref{universality_of_pullback_4}) at $(p,m,\rho)$ is the 1-arrow $(\tau^0(\lambda_P(m)),h_2^{-2}(g^{-2}(\rho))+\tau^{-1}(\delta_P(\rho)))$ where $\tau^0$ and $\tau^{-1}$ are the components of the chain homotopy that corresponds to the morphism of additive 2-functors $\tau$ in diagram (\ref{push_down_diagram_1}). Similarly, the component of the 2-arrow obtained by composing the 2-arrows $K*\iota$ and $\gamma_1*F$ on the left and the  bottom faces of the diagram (\ref{universality_of_pullback_4}) at $(p,m,\rho)$ is the 1-arrow $(h_1^{-1}(f^{-1}(m))-h_2^{-1}(g^{-1}(m)),h_2^{-2}(g^{-2}(\rho)))$. Using the chain homotopy relations, we deduce that the component of $\Theta$ at $(p,m,\rho)$ is the 1-arrow $-\tau^{-1}(m)$.

The verification of uniqueness is cumbersome but straightforward.
\end{proof}

\bibliographystyle{plain}

\end{document}